\documentclass[a4paper]{article}

\usepackage[english]{babel}
\usepackage[utf8]{inputenc}
\usepackage{amsmath}
\usepackage{graphicx}
\usepackage[colorinlistoftodos]{todonotes}
\usepackage{amsthm}
\newtheorem{thm}{Theorem}

\newtheorem{prop}[thm]{Proposition}

\newtheorem{lemma}[thm]{Lemma}
\newtheorem{defn}[thm]{Definition}
\usepackage{amssymb}
\usepackage{ esint }
\usepackage{ mathrsfs }
\usepackage{breqn}
\usepackage[hmargin=1in,vmargin=1in]{geometry}
\usepackage{amsmath}
\usepackage{amsthm}
\usepackage{graphicx}
\usepackage{placeins}

\usepackage{amsthm,amsfonts,amssymb,euscript,verbatim,bbm}

\usepackage{graphics}
\usepackage{enumerate}

\usepackage{color}

\newcommand{\eqdef}{\overset{\mbox{\tiny{def}}}{=}}

\newcommand{\pel}{p}

\def\vh {\hat{\pel}}

\def\rd {\partial}

\def\ep {\epsilon}

\def\nab {\nabla}

\newcommand{\ba}{\begin{equation}}
\newcommand{\ea}{\end{equation}}

\newcommand{\bea}{\begin{eqnarray}}
\newcommand{\eea}{\end{eqnarray}}

\def\beaa{\begin{eqnarray*}}
\def\eeaa{\end{eqnarray*}}

\title{Three New Results on Continuation Criteria for the 3D Relativistic Vlasov-Maxwell System}

\author{Neel Patel \footnote{Department of Mathematics, University of Pennsylvania, 209 South 33rd Street, Philadelphia, PA 19104. (neelpa@sas.upenn.edu) The author would like to  gratefully acknowledge partial support from the NSF grants DMS-1500916 and DMS-1200747.}}

\begin{document}
\maketitle

\begin{abstract}
In this paper, we consider sufficient conditions, called continuation criteria, for global existence and uniqueness of classical solutions to the three-dimensional relativistic Vlasov-Maxwell system. In the compact momentum support setting, we prove that $\|p_{0}^{\frac{18}{5r} - 1+\beta}f\|_{L^{\infty}_{t}L^{r}_{x}L^{1}_{p}} \lesssim 1$, where $1\leq r \leq 2$ and $\beta >0$ is arbitrarily small, is a continuation criteria. The previously best known continuation criteria in the compact setting is $\|p_{0}^{\frac{4}{r} - 1+\beta}f\|_{L^{\infty}_{t}L^{r}_{x}L^{1}_{p}} \lesssim 1$, where $1\leq r < \infty$ and $\beta >0$ is arbitrarily small, is due to Kunze in \cite{Kunze}. Thus, our continuation criteria is an improvement in the $1\leq r \leq 2$ range. In addition, we consider also sufficient conditions for a global existence result to the three-dimensional relativistic Vlasov-Maxwell system without compact support in momentum space, building on previous work by Luk-Strain \cite{Luk-Strain}. In \cite{Luk-Strain}, it was shown that $\|p_{0}^{\theta}f\|_{L^{1}_{x}L^{1}_{p}} \lesssim 1$ is a continuation criteria for the relativistic Vlasov-Maxwell system without compact support in momentum space for $\theta > 5$. We improve this result to $\theta > 3$. We also build on another result by Luk-Strain in \cite{L-S}, in which the authors proved the existence of a global classical solution in the compact momentum support setting given the condition that there exists a two-dimensional plane on which the momentum support of the particle density remains fixed. We prove well-posedness even if the plane varies continuously in time.
\end{abstract}

\section{Introduction}

Consider a distribution of charged particles described by a non-negative density function $f: \mathbb{R}_{t} \times \mathbb{R}^{3}_{x} \times \mathbb{R}^{3}_{p} \rightarrow \mathbb{R}_{+}$ of time $t$, space $x$ and momentum $p$. The Vlasov-Maxwell system describes the evolution of the density function $f(t,x,p)$ under the influence of time-dependent vector fields $E, B :  \mathbb{R}_{t} \times \mathbb{R}^{3}_{x} \rightarrow  \mathbb{R}^{3}$. Physically, this system models the behavior of a collisionless plasma:
\bea
& &\rd_t f+\vh\cdot\nabla_x f+ (E+\vh\times B)\cdot \nabla_\pel f = 0,\label{vlasov}\\
& &\rd_t E= \nabla_x \times B- j,\quad \rd_t B=-\nabla_x\times E,\label{maxwell}\\
& &\nabla_x\cdot E=\rho,\quad \nabla_x \cdot B=0.\label{constraints}
\eea
Here the charge is
$$\rho(t,x) \eqdef 4\pi\int_{\mathbb R^{{3}}} f(t,x,\pel) d\pel,$$
and the current is given by
$$
j_i(t,x) \eqdef  4\pi \int_{\mathbb R^{{3}}} \vh_i f(t,x,\pel) d\pel, \quad i=1,..., 3
$$
\\
with initial data $(f,E,B)|_{t=0} = (f_{0}, E_{0}, B_{0})$ satisfying the time-independent equations $(3)$. In the above equations, $\hat{p} = \frac{p}{p_{0}}$ where $p_{0} = (1 + |p|^{2})^{\frac{1}{2}}$.

\subsection{Notation}

In this section, we describe the notation that will be followed in the remainder of this paper. For a scalar function, $f = f(t,x,p)$, and real numbers $1\leq s,r,q \leq \infty$ we define the following norm:
$$ \|f\|_{L^{s}([0,T); L^{r}_{x}L^{q}_{p})} \eqdef \bigg( \int_{0}^{T} \Big( \int_{\mathbb{R}^{3}} \Big( \int_{\mathbb{R}^{3}} |f|^{q} dp \Big)^{\frac{r}{q}} dx \Big)^{\frac{s}{r}}  \bigg)^{\frac{1}{s}}.$$

Next,  we define $K = E + \hat{p}\times B$ and note that $|K| \leq |E| + |B|$ since $|\hat{p}| = 1$. Using the notation from \cite{Kunze}, we define:

$$ \sigma_{-1}(t,x) \eqdef \sup\limits_{|\omega| = 1}\int_{\mathbb{R}^{3}}{\frac{f(t,x,p)}{p_{0}(1+\hat{p}\cdot\omega)} dp}. $$
Also, for use in the case where we have control of the momentum support of $f$, we define
\bea\label{P(T)}
P(t) \eqdef 2 + \sup\{p\in\mathbb{R}^{3} \big| \ \exists x\in \mathbb{R}^{3}, s\in [0,t]  \text{ such that } f(s,x,p) \neq 0\}.
\eea
The notation $ a \lesssim b$ means that there exists some positive inessential constant, $C$, such that $a \leq C b$ and $a \approx b$ means that $\frac{1}{C}b\leq  a \leq Cb$.

Next, define the integral over the space-time cone $C_{t,x}$ as follows:
\bea\label{coneintegral}
\int_{C_{t,x}} f d\sigma \eqdef \int_{0}^{t} \int_{0}^{2\pi}\int_{0}^{\pi} (t-s)^{2}\sin(\theta) f(s, x+(t-s)\omega) d\theta d\phi ds
\eea
in which $\omega = (\sin(\theta)\cos(\phi),\sin(\theta)\sin(\phi),\cos(\phi))$.

Finally, for a plane $Q \subset \mathbb{R}^{3}$ containing the origin we define the projection $ \mathbb{P}_{Q}$ to be the orthogonal projection onto the plane $Q$.

\subsection{Preliminaries}

By the method of characteristics, we obtain that the particle density is conserved over the characteristics described by the system of ordinary differential equations:
\bea\label{char1}
\frac{d {X}}{ds}(s;t,x,\pel)=\hat{V}(s;t,x,\pel),
\eea
\bea\label{char2}
\frac{dV}{ds}(s;t,x,\pel)= E(s,X(s;t,x,\pel))+\hat{V}(s;t,x,\pel)\times B(s,X(s;t,x,\pel)),
\eea
together with the conditions
\bea
X(t;t,x,\pel)=x,\quad V(t;t,x,\pel)=\pel,\label{char.data}
\eea
where $\hat{V}\eqdef \frac{V}{\sqrt{1+|V|^2}}$. Further, we also have the conservation laws:

\begin{prop}
Suppose $(f,E,B)$ is a solution to the relativistic Vlasov-Maxwell system. Then we have the following conservation laws:
\bea\label{conserve1}
\frac{1}{2}\int_{\{t\}\times\mathbb{R}^{3}} (|E|^{2}+|B|^{2}) dx + 4\pi\int_{\{t\}\times\mathbb{R}^{3}\times\mathbb{R}^{3}} p_{0}f(t,x,p) dx dp = constant
\eea
and
\bea\label{conserve2}
\|f\|_{L^{\infty}_{x,p}}(t) \leq \|f_{0}\|_{L^{\infty}_{x,p}}.
\eea

\end{prop}

Note that by interpolation, the conservation laws above imply that $\|f\|_{L^{q}_{x,p}}(t) \lesssim \|f_{0}\|_{L^{q}_{x,p}}$ for $1\leq q \leq \infty$. Thus, given sufficiently nice initial data, we can assume $L^{2}$ bounds on $K$ and $L^{q}$ bounds on $f$ for $1\leq q \leq \infty$. The Glassey-Strauss decomposition is $E = E_{0} + E_{T} + E_{S}$, where $E_{0}$ depends only on the initial data, and $E_{T}$ and $E_{S}$ are:
\bea\label{ET}
E_{T} = -\int_{|y-x|\leq t}\int_{\mathbb{R}^{3}} \frac{(\omega+\hat{p})(1-|\hat{p}|^{2})}{(1+\hat{p}\cdot \omega)^{2}} f(t-|y-x|,y,p)dp\frac{dy}{|y-x|^{2}}
\eea
\bea\label{ES}
E_{S} = -\int_{|y-x|\leq t}\int_{\mathbb{R}^{3}} \nabla_{p}\Big(\frac{(\omega+\hat{p})}{1+\hat{p}\cdot\omega}\Big) \cdot Kf dp \frac{dy}{|y-x|}
\eea
The Glassey-Strauss decomposition for the magnetic field $B = B_{0} + B_{T} + B_{S}$ is similar. Writing $K_{0} = E_{0} + \hat{p}\times B_{0}$, $K_{T} = E_{T} + \hat{p}\times B_{T}$ and $K_{S} = E_{S} + \hat{p}\times B_{S}$, we can write that $|K| \leq |E| + |B|$, $|K_{0}| \leq |E_{0}| + |B_{0}|$, $|K_{T}|\leq |E_{T}| + |B_{T}|$, and $|K_{S}|\leq |E_{S}| + |B_{S}|$, where $K_{0}$ depends only on the initial data $(f_{0}, E_{0}, B_{0})$ of the relativistic Vlasov-Maxwell system. Bounding the other terms as in Propositions 3.1 and 3.2 in \cite{L-S} we obtain:
$$|K_{T}| \lesssim \int_{|y|\leq t}\frac{dy}{|y|^{2}}\int_{\mathbb{R}^{3}}\frac{f(t-|y|,x+y,p)}{p_{0}(1+\hat{p}\cdot\omega)} dp$$
and
$$|K_{S}| \lesssim \int_{|y|\leq t}\frac{dy}{|y|}\int_{\mathbb{R}^{3}}\frac{(|E|+|B|)f(t-|y|,x+y,p)}{p_{0}(1+\hat{p}\cdot\omega)} dp$$
Recalling the definition of $\sigma_{-1}$, we can bound these expressions by:
\bea\label{KT}
|K_{T}| \lesssim \int_{|y|\leq t}\frac{\sigma_{-1}(t-|y|,x+y) dy}{|y|^{2}}
\eea
\bea\label{KS}
|K_{S}| \lesssim \int_{|y|\leq t}\frac{((|E|+|B|)\sigma_{-1})(t-|y|,x+y) dy}{|y|}
\eea
Note that the right hand side of (\ref{KS}) is in the form of $\square^{-1}(|K|\sigma_{-1})$, where $\square \eqdef \partial_{t}^{2}-\sum\limits_{i=1}^{3}{\partial_{x_{i}}^{2}}$ and $u = \square^{-1}F$ satisfies:
\bea\label{wave}
\square u = F; \  u|_{t=0} = \rd_{t}u|_{t=0}=0
\eea

\subsection{Previous Results}

Luk-Strain \cite{L-S} stated the following version of the Glassey-Strauss result in \cite{Glassey-Strauss} in the case where $f_{0}$ is compactly supported in momentum space:

\begin{thm}\label{GSresult}

Consider initial data $(f_{0}, E_{0}, B_{0})$ where $f_{0} \in H^{5}(\mathbb{R}^{3}_{x}\times\mathbb{R}^{3}_{p})$ is non-negative and has compact support in $(x,p)$, and $E_{0}, B_{0} \in H^{5}(\mathbb{R}^{3}_{x})$ such that $(3)$ holds. Suppose $(f,E,B)$ is the unique classical solution to the relativistic Vlasov-Maxwell system $(1)-(3)$ in the time interval $[0,T)$ and there exists a bounded continuous function $P: [0,T)\rightarrow \mathbb{R}_{+}$ such that 

$$f(t,x,p) = 0 \text{ for } |p|\geq P(t) \ \forall x\in \mathbb{R}, t\in [0,T).$$
Then our solution $(f,E,B)$ extends uniquely in $C^{1}$ to a larger time interval $[0, T+\epsilon]$ for some $\epsilon > 0$.
\end{thm}

Additional assumptions on the Vlasov-Maxwell system, such as the condition that $f(t,x,p) = 0 \text{ for } |p|\geq P(t) \ \forall x\in \mathbb{R}, t\in [0,T)$ in the above theorem, are known as continuation criteria as they allow us to extend the interval of existence of a solution. Luk-Strain \cite{Luk-Strain} removed the condition of compact support in momentum space and proved continuation criteria in the space $H^{D}(w_{3}(p)^{2}\mathbb{R}^{3}_{x}\times\mathbb{R}^{3}_{p})$, which is the weighted Sobolev space defined by the norm:
$$ \|f\|_{H^{D}(w_{3}(p)^{2}\mathbb{R}^{3}_{x}\times\mathbb{R}^{3}_{p}}) = \sum_{0\leq k\leq {D}}\| \left( \nab^k_{x,\pel} f \right) w_{3}\|_{L^2_x L^2_\pel}$$ where the weight is defined as $w_{3}(p) = p_{0}^{\frac{3}{2}}\log(1+p_{0})$. Luk-Strain \cite{Luk-Strain} proved the following in this weighted Sobolev space:

\begin{thm}\label{LSorigin}
Let $(f_0(x,\pel),E_0(x),B_0(x))$ be {a $3$D} initial data set which satisfies the constraints \eqref{constraints} and such that {for some $D\geq 4$}, $f_0\in H^{{D}}(w_3(\pel)^2\mathbb R^3_x\times\mathbb R^3_\pel)$ is non-negative and obeys the bounds
\bea\label{ini.bd.2.5}
\sum_{0\leq k\leq {D}}\| \left( \nab^k_{x,\pel} f_0 \right) w_{3}\|_{L^2_x L^2_\pel}<\infty,
\eea
\bea\label{ini.bd.3}
\|\int_{\mathbb R^3} \sup\{f_0(x+y,\pel+w) \pel_0^3:\, |y|+|w|\leq R\}\, d\pel\|_{L^\infty_x} \leq C_R,
\eea
\bea\label{ini.bd.4}
\|\int_{\mathbb R^3} \sup\{|\nabla_{x,\pel} f_0|(x+y,\pel+w) \pel_0^3:\, |y|+|w|\leq R\}\, d\pel\|_{L^\infty_x} \leq C_R,
\eea
\begin{equation}\label{ini.bd.5}
\|\int_{\mathbb R^3}\sup\{|\nabla_{x,\pel} f_0|^2(x+y,\pel+w) w_{3}^2 :\,|y|+|w|\leq R\}\, d\pel\|_{L^\infty_x}
  \leq C_R^2,
\end{equation}
and
\bea\label{ini.bd.5.5}
\|\int_{\mathbb R^3}\sup\{|\nabla_{x,\pel}^2 f_0|(x+y,\pel+w) \pel_0:\,|y|+|w|\leq R\} d\pel\|_{L^\infty_x}  \leq C_R,
\eea
for some different constants $C_R <\infty$ for every $R>0$; and the initial electromagnetic fields $E_0, B_0 \in H^{{D}}(\mathbb R^3_x)$ obey the bounds
\bea\label{ini.bd.6}
\sum_{0\leq k\leq {{D}}}(\|\nab_x^k E_0 \|_{L^2_x}+\|\nab_x^k B_0 \|_{L^2_x})<\infty.
\eea 

Given this initial data set, there exists a unique local solution $(f,E,B)$ on some $[0,T_{loc}]$ such that $f_{0} \in L^{\infty}([0,T_{loc}]; H^{4}(w_{3}(p)^{2}\mathbb{R}^{3}_{x}\times\mathbb{R}^{3}_{p}) $ and $E, B \in L^{\infty}([0,T_{loc}]; H^{4}(\mathbb{R}^{3}_{x}))$.

Let $(f,E,B)$ be the unique solution to \eqref{vlasov}-\eqref{constraints} in $[0,T_*)$. Assume that 
\bea\label{ini.bd.10}
sup \int_{0}^{T_{*}} \big( |E(s,X(s;t,x,p)| + |B(s,X(s;t,x,p)|  \big) ds < \infty
\eea
where the supremum is taken over all $(t,x,p) \in [0,T_{*})\times\mathbb{R}^{3}\times\mathbb{R}^{3}$. Then, there exists $\epsilon>0$ such that the solution extends uniquely beyond $T_*$ to an interval $[0,T_*+\epsilon]$ such that $E,B\in L^\infty([0,T_*+\ep];H^{{D}}(\mathbb R^3_x))$ and  $f\in L^\infty([0,T_*+\ep];H^{{D}}(w_3(\pel)^2 d\pel ~dx ))$.
\end{thm}

Under the additional assumption that $\|p_{0}^{N}f_{0}\|_{L^{\infty}_{t}([0,T_{*});L^{1}_{x}L^{1}_{p})} \lesssim C_{N}$ for a large positive integer $N = N_{\theta}$ depending on $\theta$, Luk-Strain \cite{Luk-Strain} use the above theorem to prove that $\|p_{0}^{\theta}f\|_{L^{\infty}_{t}([0,T_{*});L^{1}_{x}L^{1}_{p})} \lesssim 1$ is a continuation criteria for the relativistic Vlasov-Maxwell system without compact support in momentum space for $\theta > 5$. To do so, Luk-Strain \cite{Luk-Strain} utilized Strichartz estimates on both the $K_{T}$ and $K_{S}$ bounds and interpolation inequalities. We note in this paper that we only need the initial data assumption that $\|p_{0}^{N}f_{0}\|_{L^{\infty}_{t}([0,T_{*});L^{1}_{x}L^{1}_{p})} \lesssim 1$ for some $N > 5$.
Finally, for comparison to the results in this paper, we also present the result in the compact support setting due to Kunze in \cite{Kunze}:

\begin{thm}\label{KunzeThm}
Suppose we have initial data $f_{0}\in C^{1}_{0}(\mathbb{R}^{3}\times\mathbb{R}^{3})$ and $E_{0}, B_{0} \in C^{2}_{0}(\mathbb{R}^{3})$ satisfying the constraints (\ref{constraints}). Let $(f,E,B)$ be the unique solution to (\ref{vlasov})-(\ref{constraints}) in the time interval $[0,T_{*}]$. If
$$ \|p_{0}^{\frac{4}{r}-1+\beta}f\|_{L^{\infty}([0,T_{*}];L^{r}_{x}L^{1}_{p})}$$ for some $1\leq q < \infty$ and $\beta > 0$, then we can continuously extend our solution $(f,E,B)$ uniquely to an interval $[0,T_{*}+\epsilon]$.
\end{thm}

The final result of this paper builds on the work of Luk-Strain\cite{L-S} in the compact setting:

\begin{thm}\label{planesuppLS}
Consider initial data $(f_{0}, E_{0}, B_{0})$ where $f_{0} \in H^{5}(\mathbb{R}^{3}_{x}\times\mathbb{R}^{3}_{p})$ is non-negative and has compact support in $(x,p)$, and $E_{0}, B_{0} \in H^{5}(\mathbb{R}^{3}_{x})$ such that $(3)$ holds. Suppose $(f,E,B)$ is the unique classical solution to the relativistic Vlasov-Maxwell system $(1)-(3)$ in the time interval $[0,T)$. Assume that there exists a plane $Q \subset \mathbb{R}^{3}$ with $0\in Q$ and a bounded continuous function $\kappa :[0,T_{+})\rightarrow \mathbb{R}^{3}$ such that 

$$ f(t,x,p) = 0 \text{ \ for \ } |\mathbb{P}_{Q}p| \geq \kappa(t), \ \forall x\in\mathbb{R}^{3}.$$ Then there exists an $\epsilon > 0$ such that the solution extends uniquely in $C^{1}$ to a larger time interval $[0,T+\epsilon]$.
\end{thm}

\subsection{Main Results}

We extend the result of Luk-Strain \cite{Luk-Strain} to the case where $\theta > 3$ in Theorem \ref{MAIN} below. Note that we also remove the $\theta$ dependence of $N$ in the moment bound, $\|p_{0}^{N}f_{0}\|_{L^{\infty}_{t}([0,T_{*});L^{1}_{x}L^{1}_{p})} \lesssim C_{N}$, of the initial data.

\begin{thm}\label{MAIN}
Consider initial data $(f_{0}, E_{0}, B_{0})$ satisfying (\ref{ini.bd.2.5})-(\ref{ini.bd.6}) and the additional condition that $\|p_{0}^{\tilde{N}}f_{0}\|_{L^{\infty}_{t}([0,T_{*});L^{1}_{x}L^{1}_{p})} \lesssim C_{N}$ for some $\tilde{N} >5$. Let $(f,E,B)$ be the unique solution to $(1)-(3)$ in $[0,T)$  and assume that
$$ \|p_{0}^{\theta}f\|_{L^{1}_{x}L^{1}_{p}}(t) \leq A(t) $$
for some $\theta > 3$ and some bounded continuous function $A: [0,T) \rightarrow \mathbb{R}_{+}$. Then we can extend our solution $(f,E,B)$ uniquely to an interval $[0,T+\epsilon]$ such that $E,B\in L^\infty([0,T+\ep];H^{{D}}(\mathbb R^3_x))$ and  $f\in L^\infty([0,T+\ep];H^{{D}}(w_3(\pel)^2 d\pel ~dx ))$.
\end{thm}

\subsubsection{Outline of Proof}

The key to our proof is to gain bounds on $\|p_{0}^{N}f\|_{L^{\infty}_{t}([0,T);L^{1}_{x}L^{1}_{p})}$ for a power of $N>5$ since we later prove that the expression in (\ref{ini.bd.10}) can be bounded by $\|p_{0}^{N}f\|_{L^{1}_{x}L^{1}_{p}}$ where $N = 5 + \lambda$ for any $\lambda > 0$. As proven in Proposition 7.3 in Luk-Strain \cite{Luk-Strain}, we have the following standard moment estimate for $N > 0$:
\bea\label{momentest}
 \|p_{0}^{N}f\|_{L^{\infty}_{t}([0,T);L^{1}_{x}L^{1}_{p})} \lesssim   \|p_{0}^{N}f_{0}\|_{L^{1}_{x}L^{1}_{p}} + \|E\|_{L^{1}_{t}([0,T);L^{N+3}_{x})}^{N+3} + \|B\|_{L^{1}_{t}([0,T);L^{N+3}_{x})}^{N+3}
\eea

Assume $N > 3$. Our goal now is to bound the terms $\|E\|_{L^{1}_{t}([0,T);L^{N+3}_{x})}^{N+3}$ and $\|B\|_{L^{1}_{t}([0,T);L^{N+3}_{x})}^{N+3}$ on the right hand side by $\|p_{0}^{N}f\|_{L^{\infty}_{t}([0,T);L^{1}_{x}L^{1}_{p})}^{\alpha} $ for some $\alpha < 1$. To do so, we employ the Glassey-Strauss decomposition of the field term 
\bea\label{altdec}
\tilde{K} \eqdef (E,B) = (E_{0},B_{0}) +  (E_{T},B_{T}) + (E_{S},B_{S})
\eea
where $E_{0}$ and $B_{0}$ depend only on the initial data of our system. The terms on the right hand side of (\ref{altdec}) have the same bounds as the $K_{T}$ and $K_{S}$ bounds in (\ref{KT}) and (\ref{KS}) respectively. To bound the $K_{T}$ term, we utilize estimates for the averaging operator on the sphere and then apply the interpolation inequality used in Luk-Strain \cite{Luk-Strain}. To do so, we define the operator $$W_{\alpha}(h(t,x))\eqdef \int_{0}^{t}{s^{2-\alpha}\fint_{\mathbb{S}^{2}}{h(t-s,x+s\omega)d\mu(\omega)}ds}, $$ where $\fint_{X} g(x) d\mu(x)$ denotes the average value of the function $g$ over the measure space $(X,\mu)$. We can bound $K_{T}$ with this operator setting $\alpha =2$:
$$ K_{T} \lesssim W_{2}(\sigma_{-1}).$$ Thus, using a known averaging operator estimate and interpolation inequalities, we obtain the following bound on $K_{T}$:
$$\|K_{T}\|_{L^{r}_{t}([0,T];L^{N+3}_{x})} ^{N+3} \lesssim \|p_{0}^{\frac{N(1+\gamma)+3+\delta}{2+\gamma}}f\|_{L^{\infty}_{t}([0,T);L^{1}_{x}L^{1}_{p})}^{2+\gamma}\|p_{0}^{N}f\|_{L^{\infty}_{t}([0,T);L^{1}_{x}L^{1}_{p})}^{1-\gamma} $$ for $1\leq r \leq \infty$, $\gamma \in (0,1)$, $N>3$ and $\delta > 0$.
To bound the $K_{S}$ term, we apply Strichartz estimates for the wave equation and utilize the method from Sogge \cite{Sogge} as used in Kunze \cite{Kunze}. This method requires us to use the assumption that $$\|\sigma_{-1}\|_{L^{\infty}_{t}([0,T];L^{2}_{x})} \lesssim 1.$$ We apply wave equation Strichartz estimates on a partition of the interval $[0,T) = \cup_{i=1}^{k-1} [T_{i},T_{i+1}]$ such that the quantity $\|\sigma_{-1}\|_{L^{\infty}_{t}([T_{i},T_{i+1}];L^{2}_{x})}$ is sufficiently small for us to use an iteration scheme to bound $K_{S}$ over the interval $[0,T]$.

In Luk-Strain \cite{Luk-Strain}, the $K_{T}$ term was bounded by using H\"older's inequality to rewrite the bound (\ref{KS}) in the form of a solution to the wave equation described by (\ref{wave}). Then, they used Strichartz estimates for the wave equation. Instead, we use a more direct approach by using averaging operator estimates. This approach also enables us to preserve the singularity in the $\sigma_{-1}$ denominator, which is useful in reducing the power of $p_{0}$ in the bound of $K_{T}$.
By the bounds on $K_{T}$ and $K_{S}$, we obtain a bound on $\|K\|_{L^{1}_{t}([0,T);L^{N+3}_{x})}^{N+3}$ for some $\gamma \in (0,1)$:

$$ \|K\|_{L^{1}_{t}([0,T);L^{N+3}_{x})}^{N+3} \lesssim 1 + \|p_{0}^{\frac{N(1+\gamma)+3+\delta}{2+\gamma}}f\|_{L^{\infty}_{t}([0,T);L^{1}_{x}L^{1}_{p})}^{2+\gamma}\|p_{0}^{N}f\|_{L^{\infty}_{t}([0,T);L^{1}_{x}L^{1}_{p})}^{1-\gamma}$$ where the implicit constant in this inequality also depends on the quantity $\|\sigma_{-1}\|_{L^{\infty}_{t}([0,T];L^{2}_{x})}$.

Thus, assuming that $\|p_{0}^{\frac{N(1+\gamma)+3+\delta}{2+\gamma}}f\|_{L^{\infty}_{t}([0,T);L^{1}_{x}L^{1}_{p})} \lesssim 1$, we can insert this estimate into the standard moment estimate above to gain a higher moment bound $\|p_{0}^{N}f\|_{L^{\infty}_{t}([0,T);L^{1}_{x}L^{1}_{p})} \lesssim 1$. By an iteration of this process, we eventually arrive at the bound $\|p_{0}^{\hat{N}}f\|_{L^{\infty}_{t}([0,T);L^{1}_{x}L^{1}_{p})} \lesssim 1$ for some $\hat{N} > 5$, which proves the result of Theorem \ref{MAIN}. Our proof of Theorem \ref{MAIN} relies on this new method of incrementally using lower moment bounds to gain control over slightly higher moment bounds, as compared to directly bounding all arbitrarily large moments by some fixed small moment.

Kunze \cite{Kunze} proves this result in his paper with the assumption of initial compact support in the momentum variable. This method allows him to save an entire power of $p_{0}$ and use a Gronwall-type inequality to bound the momentum support at time $T$. We do not have this extra control given by the momentum support of $f$ and needed a wider range of bounds on the $K_{T}$ term. Actually, our method for bounding $K_{T}$ can give us strictly better bounds than those of Kunze \cite{Kunze}. In \cite{Kunze}, for $2\leq r < 6$, Kunze proves the bound:
\bea
\|K_{T}\|_{L^{r}_{t}([0,T];L^{r}_{x})}\leq C_{T} \|\sigma_{-1}\|_{L^{\infty}_{t}([0,T];L^{2}_{x})}.
\eea

In comparison, we prove the bound in (\ref{OpEst}) and Proposition \ref{newKTest} which lowers the Lebesgue exponent of the norm on $\sigma_{-1}$: $$\|K_{T}\|_{L^{\infty}_{t}L^{mq}_{x}} \lesssim \|\sigma_{-1}\|_{L^{\infty}_{t}L^{q}_{x}}$$ for $1\leq m \leq 3$, $q>3-\frac{3}{m}$ and $\frac{3m-1}{2m} \leq q \leq \infty$. For the purposes of this problem, obtaining lower Lebesgue exponents yields better estimates because by interpolation $$\|\sigma_{-1}\|_{L^{\infty}_{t}([0,T];L^{q}_{x})} \lesssim \|p_{0}^{2q-1+\nu}f\|_{L^{1}_{x,p}}$$ for some $\nu > 0$. Thus, lower powers of $q$ yield lower powers of $p_{0}$ on the right hand side. This allows us to bound Lebesgue norms of $K$ by lower moments (i.e. lower powers of $p_{0}$), which gives us better control on $K$. We use this extra control on $\|K_{T}\|_{L^{r}_{x}}$ in the range $2 \leq r \leq 6$ by lower moments to help us in the case of initial data with compact support, in which we prove the following:

\begin{thm}\label{main3}
Consider initial data $(f_{0},E_{0}, B_{0})$ satisfying the conditions in Theorem \ref{GSresult} and $(f,E,B)$ is the unique classical solution to (\ref{vlasov})-(\ref{constraints}) in the interval $[0,T)$. Suppose we impose the additional assumption that
\bea\label{compactcriteria}
\|p_{0}^{\frac{18}{5r} - 1+\beta}f\|_{L^{\infty}_{t}L^{r}_{x}L^{1}_{p}} \lesssim 1
\eea
for some $1\leq r \leq 2$ and some $\beta >0$ arbitrarily small. Then, we can continuously extend our solution $(f,E,B)$ to an interval $[0,T+\epsilon]$ in $C^{1}$ for some $\epsilon > 0$.
\end{thm}

The exponent of $p_{0}$ in (\ref{compactcriteria}) is strictly better than the exponent found in the result stated in Theorem \ref{KunzeThm} since $\frac{18}{5r}-1+\beta < \frac{4}{r}-1+\beta$. For example, if $r =1$, our criteria is $\|p_{0}^{\frac{13}{5}+\beta}f\|_{L^{\infty}_{t}L^{1}_{x}L^{1}_{p}} \lesssim 1$ which is better than the known criteria of $\|p_{0}^{3+\beta}f\|_{L^{\infty}_{t}L^{1}_{x}L^{1}_{p}} \lesssim 1$ due to \cite{Kunze}. Similarly for the $r=2$ case, our criteria is $\|p_{0}^{\frac{4}{5}+\beta}f\|_{L^{\infty}_{t}L^{2}_{x}L^{1}_{p}} \lesssim 1$ which is better than the known criteria of $\|p_{0}f\|_{L^{\infty}_{t}L^{2}_{x}L^{1}_{p}} \lesssim 1$ due to \cite{Kunze}.

\subsubsection{Outline of Proof}

The first step to proving Theorem \ref{main3} is to utilize the decomposition 
\bea
|E| + |B| \leq |E_{0}| + |B_{0}| + |E_{T}| + |B_{T}| + |E_{S,1}| + |B_{S,1}| + |E_{S,2}| + |B_{S,2}|
\eea
as in Luk-Strain \cite{L-S}. The advantage to this decomposition is that it allows us to utilize the conservation of the $L^{2}$ norm of $|K_{g}| = (|E\cdot \omega|^{2} + |B \cdot \omega|^{2} + |E - \omega \times B|^{2} + |B + \omega\times E|^{2})^{\frac{1}{2}}$ on the space-time cone and it also reduces the power of $1+\hat{p}\cdot \omega$ in the $K_{S,1}$ term by a power of $\frac{1}{2}$. This decomposition of $K_{S}$ into the two terms $K_{S,1}$ and $K_{S,2}$ allows us to gain better bounds on the $K_{S}$ part of the field decomposition. We can bound each element of this decomposition as follows:
$$|K_{T}| = |E_{T}| + |B_{T}| \lesssim W_{2}(\sigma_{-1})$$

$$|K_{S,1}| = |E_{S,1}| + |B_{S,1}| \lesssim \square^{-1} (|K| \Phi_{-1})$$

$$|K_{S,2}| = |E_{S,2}| + |B_{S,2}| \lesssim (W_{2}(\sigma_{-1}^{2}))^{\frac{1}{2}}$$
where
$$\Phi_{-1}(t,x) \eqdef \max\limits_{|\omega|=1}\int_{\mathbb{R}^{3}} \frac{f(t,x,p)dp}{p_{0}(1+\hat{p}\cdot\omega)^{\frac{1}{2}}}.$$

As in the proof of Theorem \ref{MAIN}, we can apply averaging operator estimates to the $K_{T}$ and $K_{S,2}$ to get the bounds
$$ \|K_{S,2}\|_{L^{\infty}_{t}L^{2mq}_{x}} \lesssim \|\sigma_{-1}\|_{L^{\infty}_{t}L^{2q}_{x}}$$
and
$$\|K_{T}\|_{L^{\infty}_{t}L^{mq}_{x}} \lesssim \|\sigma_{-1}\|_{L^{\infty}_{t}L^{q}_{x}}$$
where $q > 3-\frac{3}{m}$ and $\frac{3m-1}{2m} \leq q \leq \infty$ for $1\leq m \leq 3$. (Note that this is where we will use the improved estimate on $\|K_{T}\|_{L^{r}_{x}}$ in the range $2 \leq r \leq 6$, which also give us bounds on the $K_{S,2}$ term. Specifically, in this paper, we use the exponent $r = 4 + \delta$ for some $\delta > 0$ appropriately small.) Using these estimates and using Strichartz estimates, we apply similar techniques as in the proof of Theorem \ref{MAIN} to the $K_{S,1}$ term to obtain bounds on $K$. We are given better control of the $K_{S,1}$ term because in the inequality $|K_{S,1}|\lesssim \square^{-1}(|K|\Phi_{-1})$, $\Phi_{-1}$ has a lower power of singularity in the denominator than $\sigma_{-1}$. (We partition our time interval $[0,T]$ under the assumption that $\|\Phi_{-1}\|_{L^{\infty}_{t}([0,T];L^{2}_{x})} \lesssim 1$, which is a weaker assumption that $\|\sigma_{-1}\|_{L^{\infty}_{t}([0,T];L^{2}_{x})}\lesssim 1$.)

The goal of our bound on $K$ in this proof is not to gain bounds on higher moments, as in the proof of Theorem \ref{MAIN}. Instead, we use an idea of Pallard \cite{refinedPallard} and bound the integral of the electric field over the characteristics by appropriate Lebesgue norms involving $f$ and $K$:
$$|P(T)| \lesssim 1 + \|\sigma_{-1}\|_{L^{\infty}_{t}L^{3+}_{x}} + \|\sigma_{-1}|K|\ln^{\frac{1}{3}}(1+P(t))\|_{L^{1}_{t}L^{\frac{3}{2}}_{x} ([0,T]\times \mathbb{R}^{3})}.$$
Using the bounds on $\|K\|_{L^{\infty}_{t}([0,T);L^{r}_{x})}$ for some exponent $r > 4$ appropriately close to $4$ and interpolation inequalities, we can then bound these terms by powers of $P(T)$ smaller than $1$ to obtain an inequality of the form:
$$ P(T) \lesssim 1 + P(T)^{\gamma}\ln^{\lambda}(P(t)) $$ for some $\gamma \in [0,1)$ and $\lambda > 0$. From here, we conclude that $P(T)\lesssim  1$.
\\
\par Our final result improves the continuation criteria due to Luk-Strain in \cite{L-S}. First, consider a family of planes $\{Q(t)\}_{t\in [0,T]}$. At $t=0$, we choose a normal vector $n_{3}(0)$ orthogonal to the plane $Q(0)$ at the origin.

\begin{defn}
	A family of planes $\{Q(t)\}_{t\in [0,T]}$ containing the origin is considered to be \textbf{uniformly continuous family of planes} in the following sense:
	There exists a partition $[T_{i}, T_{i+1})$ of $[0,T)$ such that locally in a small time interval, for say $ s\in[T_{i},T_{i+1})$, we can let $n_{3}(s)$ be the normal to $Q(s)$ at the origin that is on the same half of $\mathbb{R}^{3}$ as a $n_{3}(T_{i})$, meaning $\angle(n_{3}(s),n_{3}(T_{i})) < \angle(n_{3}(s),-n_{3}(T_{i}))$, where $\angle(v,w)\eqdef \cos^{-1}\big(\frac{v\cdot w}{|v||w|}\big)$. Then, the map $n_{3}: [0,T)\rightarrow \mathbb{S}^{2}$ is uniformly continuous.
\end{defn}

Using this definition, we prove the following:
\begin{thm}\label{main1}
	Suppose $f_{0}(x,p) \in H^{5}(\mathbb{R}^{3}\times\mathbb{R}^{3})$ with compact support in $(x,p)$, $E_{0}, B_{0}\in H^{5}(\mathbb{R}^{3})$. Let $(f,E,B)$ be the classical solution in $L^{\infty}_{t}([0,T);H^{5}_{x,p}) \times L^{\infty}_{t}([0,T);H^{5}_{x}) \times L^{\infty}_{t}([0,T);H^{5}_{x})$ to the Vlasov-Maxwell system in $[0,T)$. Let $\{Q(t)\}$ be a uniformly continuous family of planes containing the origin such that there exists a bounded, continuous function $\kappa: [0,T)\rightarrow \mathbb{R}_{+}$ such that
	
	$$f(t,x,p) = 0 \text{ for } |\mathbb{P}_{Q(t)}p|\geq \kappa(t) \ \forall x\in \mathbb{R}$$
	
	Then there exists $\epsilon > 0$ such that our solution can be extended continuously in time in $H^{5}$ to $[0, T+\epsilon]$.
\end{thm}

A more general theorem can be proven. Theorem \ref{main1} will be a special case of this theorem. First, we need to define a time dependent coordinate system on $\mathbb{R}^{3}$ which will depend on the plane $Q(t)$. Let $\{n_{1}(t),n_{2}(t),n_{3}(t)\}$ be unit vectors such that $\{n_{1}(t),n_{2}(t)\}$ span $Q(t)$ and $n_{3}(t)$ is the unit normal to $Q(t)$ as defined earlier.

%Since $Q(s)$ is uniformly continuous, locally in a small time interval, for say $ s\in[s_{1},s_{2}]$, we can let $n_{3}(s)$ be the normal to $Q(s)$ that is on the same half of $\mathbb{R}^{3}$ as a $n_{3}(s_{1})$, meaning $\angle(n_{3}(s),n_{3}(s_{1})) < \angle(n_{3}(s),-n_{3}(s_{1}))$.

\par Fix a time $t\in [0,T)$. By uniform continuity of $n_{3}(t)$, there exists a partition of $[0,t)= \cup_{i=0}^{n_{t}}[T_{i},T_{i+1})$ (the number of intervals in the partition $n_{t}$ depends on $t$ and $T_{n_{t}+1}=t$) such that for $s\in[T_{i},T_{i+1})$, we have:
\bea\label{condition1}
\angle(n_{3}(s),n_{3}(T_{i})) < \angle(-n_{3}(s),n_{3}(T_{i}))
\eea
and
\bea\label{condition2}
\angle(n_{3}(s),n_{3}(T_{i})) < \frac{P(t)^{-1}}{4}
\eea

We will use this precise partition for the proof of Theorem \ref{main2} in this paper.
\begin{thm}\label{main2}
	Suppose $f_{0}(x,p) \in H^{5}(\mathbb{R}^{3}\times\mathbb{R}^{3})$ with compact support in $(x,p)$, $E_{0}, B_{0}\in H^{5}(\mathbb{R}^{3})$. Let $(f,E,B)$ be the classical solution in $L^{\infty}_{t}([0,T);H^{5}_{x,p}) \times L^{\infty}_{t}([0,T);H^{5}_{x}) \times L^{\infty}_{t}([0,T);H^{5}_{x})$ to the Vlasov-Maxwell system in $[0,T)$. Let $\{Q(t)\}$ be a uniformly continuous family of planes containing the origin. Suppose for each $t\in [0,T)$, there exists a measurable, positive function $\kappa: [0,T)\times[0,2\pi]\rightarrow \mathbb{R}_{+}$ such that $\kappa(t,\gamma) >1$,
	
	$$\text{sup}\{|\mathbb{P}_{Q(t)}p| : \frac{p\cdot n_{2}(t)}{p\cdot n_{1}(t)} = tan(\gamma), f(t,x,p)\neq 0 \text{ for some } x\in \mathbb{R}^{3}\} < \kappa(t,\gamma)$$ and
	
	$$\int_{0}^{T}{\Big(A(t)^{2}+(\int_{0}^{t}{A(s)^{8}ds})^{\frac{1}{2}}\Big)} dt < +\infty \text{ where } A(t) = \|\kappa(t,\cdot)\|_{L^{4}_{\gamma}}$$
	Then there exists $\epsilon > 0$ such that our solution can be extended continuously in time to $[0, T+\epsilon]$.
\end{thm}

Note that $\gamma$ depends on $p\in\mathbb{R}^{3}$, so we actually have $\tan(\gamma) = \tan(\gamma(p)) =  \frac{p\cdot n_{2}(t)}{p\cdot n_{1}(t)}$.

%If we let $\kappa(t,\gamma)$ be the bounded, continuous function $\kappa(t)$ from Theorem 1, then Theorem 2 gives us the result of Theorem 1. Now, for lines $a$ and $b$, let $v_{a}$ and $v_{b}$ denote unit vectors along $a$ and $b$ respectively such that $\angle(v_{a},v_{b})\leq \angle(v_{a}, -v_{b})$.

%\begin{cor}
%Let the initial data assumptions from Theorem 1 hold. Let $\{(a(t), b(t))\}$ be a uniformly continuous family of lines such that $\angle(v_{a(t)},v_{b(t)}) = \alpha$ for some constant $\alpha$ independent of $t$ and there exists a bounded, continuous function $\kappa: [0,T)\rightarrow \mathbb{R}_{+}$ such that

%$$f(t,x,p) = 0 \text{ for } |\mathbb{P}_{a(t)}p|\geq \kappa_{0}(t) \ \forall x\in \mathbb{R}$$ and

%$$f(t,x,p) = 0 \text{ for } |\mathbb{P}_{b(t)}p|\geq \kappa_{0}(t) \ \forall x\in \mathbb{R}$$

%Then there exists $\epsilon > 0$ such that our solution can be extended continuously in time to $[0, T+\epsilon]$.
%\end{cor}

\subsubsection{Outline of Proof}

We modify methods used in \cite{L-S} to prove Theorem \ref{main2}. We wish to show that the quantity $$P(t) = 2 + \text{sup}\{|p| : f(s,x,p)\neq 0 \text{ for some } 0\leq s \leq t \text{ and } x\in\mathbb{R}^{3}\}$$ is bounded on $[0,T)$. By the method of characteristics (see \cite{L-S}), we have the bound 

$$P(t)\lesssim 1+\sup\limits_{(t,x,p)\in\mathbb{R}\times\mathbb{R}^{3}\times\mathbb{R}^{3}}\int_{0}^{t}{|E(s;X(s;t,x,p))|+|B(s;X(s;t,x,p))|ds}$$. 

We wish to bound the momentum support quantity $P(t)$. To do so, we first find appropriate estimates on $E$ and $B$. We again use the decomposition:

\begin{align*} 
4\pi E(x,t) &= (E)_{0} + E_{S,1} + E_{S,2} + E_{T}  \\ 
4\pi B(x,t) &= (B)_{0} + B_{S,1} + E_{S,2} + B_{T}
\end{align*}

where $(E)_{0}$ and $(B)_{0}$ depend only on the initial data. We have the following estimates from Proposition 3.1 and Proposition 3.4 in \cite{L-S}:

\bea\label{ET0}
|E_{T}(t,x)| + |B_{T}(t,x)|  &\lesssim \int_{C_{t,x}}{\int_{\mathbb{R}^{3}}{\frac{f(s,x+(t-s)\omega,p)}{(t-s)^{2}p_{0}^{2}(1+\hat{p}\cdot\omega)^{\frac{3}{2}}}dp \ d\omega}}
\eea
\bea\label{ES1}
|E_{S,1}(t,x)| + |B_{S,1}(t,x)|  &\lesssim \int_{C_{t,x}}{\int_{\mathbb{R}^{3}}{|B|\frac{f(s,x+(t-s)\omega,p)}{(t-s)p_{0}(1+\hat{p}\cdot\omega)^{\frac{1}{2}}}dp \ d\omega}} 
\eea
\bea\label{ES2}
|E_{S,2}(t,x)| + |B_{S,2}(t,x)|  &\lesssim \int_{C_{t,x}}{\int_{\mathbb{R}^{3}}{\frac{(|E\cdot\omega|+|B\cdot\omega|+|B+\omega\times E|)f(s,x+(t-s)\omega,p)}{(t-s)p_{0}(1+\hat{p}\cdot\omega)} dp \ d\omega}} 
\eea

Next, we prove analogous bounds on the momentum integral $\int_{\mathbb{R}^{3}}{\frac{f(s,x+(t-s)\omega,p)}{(t-s)^{2}p_{0}^{2}(1+\hat{p}\cdot\omega)^{\frac{3}{2}}}dp}$ as those found in \cite{L-S}. Partitioning the time interval $[0,T)$ into subintervals $[T_{i}, T_{i+1}]$ small enough as described in (\ref{condition1}) and (\ref{condition2}). Applying these two conditions to bound (\ref{ET0}), (\ref{ES1}) and (\ref{ES2}) on each subinterval $[T_{i}, T_{i+1}]$ in an analogous method to \cite{L-S}. (These conditions allow us to approximate the integrals in each time subinterval by the integral at the endpoints, as the variation in the momentum support plane is very small on each subinterval. This observation is the key to proving Theorem \ref{main2}.) We then sum over the partition to prove analogous bounds on the field terms to those found in \cite{L-S}. From here, we conclude that $P(T)\lesssim 1$ by the bootstrap argument of \cite{L-S}.

\section{The operator $W_{\alpha}$}

The average value of a function $g : X\rightarrow \mathbb{R}$ defined on a space $X$ with finite measure $\nu(X)$ is denoted by
\bea
\fint_{X}{g(x)dx} = \frac{1}{\nu(X)}\int_{X}{g(x) dx}
\eea

Define the operator $W_{\alpha}$ by 
\bea\label{operator0}
W_{\alpha}(h(t,x)) = \int_{|x-y|\leq t}{\frac{h(t-|x-y|,y)}{|x-y|^{\alpha}}dy} = \int_{0}^{t}{s^{2-\alpha}\fint_{\mathbb{S}^{2}}{h(t-s,x+s\omega)d\mu(\omega)}ds}
\eea
where $d\mu(\omega)$ is the spherical measure. Then, by (\ref{KT}) we obtain the following:

\begin{prop}\label{FieldEst}
For the electric and magnetic fields, we have the estimate:
\begin{align*} 
|K_{T}(t,x)|  &\lesssim W_{2}(\sigma_{-1})
\end{align*}

\end{prop}

Thus, we wish to obtain estimates for the operator $W_{\alpha}$ for $\alpha = 2$. We prove an estimate for general $\alpha$. Consider:
\bea\label{operator1}
(T_{\alpha,s}h)(t,sx) = s^{2-\alpha}\fint_{\mathbb{S}^{2}}{h(t-s,sx+s\omega)d\mu(\omega)}
\eea

A Schwartz function is a $C^{\infty}$ function $f:\mathbb{R}^{n} \rightarrow \mathbb{R}$ such that for any pair of multi-indices, $\alpha$ and $\beta$, there exists a finite constant $C_{\alpha,\beta}$ satisfying $\sup\limits_{x\in\mathbb{R}^{n}}|x^{\alpha}\partial^{\beta}f(x)| \leq C_{\alpha,\beta}$. The set of Schwartz functions form a vector space called the Schwartz space, which is dense in the space $L^{q}$ for $1\leq q < \infty$. On the Schwartz space, denoted by $\mathscr{S}(\mathbb{R}^{n})$, we have known estimates for the averaging operator $Af = \fint_{\mathbb{S}^{2}}{f(x+\omega)d\mu(\omega)}$ from (2) in \cite{averaging}:

\begin{thm}
The estimate
\bea\label{average}
\|Af \|_{L^{a}} \lesssim \|f\|_{L^{q}}, \ f\in\mathscr{S}(\mathbb{R}^{n})
\eea
holds if and only if $(\frac{1}{q}, \frac{1}{a})$ in the convex hull of $(0,0)$, $(1,1)$, and $(\frac{n}{n+1}, \frac{1}{n+1})$.
\end{thm}

For the case $n=3$, the inequality (\ref{average}) holds for $(\frac{1}{q}, \frac{1}{a})$ in the convex hull of $\{ (0,0), (1,1), (\frac{3}{4}, \frac{1}{4})\}$. Thus, setting $a=mq$ for $1\leq m \leq 3$ and for a range of $q$ to be calculated, we have that:
\bea\label{operator2}
\|(T_{\alpha,s}h)(t,sx)\|_{L^{mq}_{x}} \lesssim s^{2-\alpha} \|h(t-s,sx)\|_{L^{q}_{x}}
\eea

After a change of variables in the spatial coordinates,

$$\|T_{\alpha,s}h(t)\|_{L^{mq}_{x}} \lesssim s^{2-\alpha+\frac{3}{mq}-\frac{3}{q}} \|h(t-s)\|_{L^{q}_{x}}$$

Applying this estimate to the operator $W_{\alpha}$ under the additional assumption $2-\alpha+\frac{3}{mq}-\frac{3}{q} > -1$,
\begin{equation*}
\begin{split}
\|W_{\alpha}h(t)\|_{L^{r}_{t}L^{mq}_{x}} & \leq \left\|\int_{0}^{t}{\|T_{\alpha,s}h(t)\|_{L^{mq}_{x}} ds}\right\|_{L^{r}_{t}} \\
 & \lesssim \left\|\int_{0}^{t}{s^{2-\alpha+\frac{3}{mq}-\frac{3}{q}}\|h(t-s)\|_{L^{q}_{x}} ds}\right\|_{L^{r}_{t}} \\
 & \lesssim \left\|\int_{0}^{t}{s^{2-\alpha+\frac{3}{mq}-\frac{3}{q}}\|h\|_{L^{\infty}_{t}L^{q}_{x}} ds}\right\|_{L^{r}_{t}} \\
 & \lesssim \|h\|_{L^{\infty}_{t}L^{q}_{x}}
\end{split}
\end{equation*}
where $1\leq r \leq \infty$, $t\in [0,T]$, $x\in \mathbb{R}^{3}$, and the implicit constant in the upper bound is a continuous function of $T$, $r$ and $\alpha$. It remains to check the range of $q$ for which $(\frac{1}{q},\frac{1}{mq})$ lies in the convex hull described above.

We observe that the line connecting the points $(x,y)= (1,1)$ and $(x,y) = (\frac{3}{4},\frac{1}{4})$ is represented by the equation $y=3x -2$. Thus, the line $y=\frac{1}{m}x$ meets the line $y=3x-2$ when $x=\frac{2m}{3m-1}$.
\\
\\
Summarizing:

\begin{lemma}
For $1\leq r \leq \infty$, $1\leq m \leq 3$, $2-\alpha+\frac{3}{mq}-\frac{3}{q} > -1$, $\frac{3m-1}{2m} \leq q \leq \infty$,
\bea\label{OpEst}
\|W_{\alpha}h(t,x)\|_{L^{r}_{t}([0,T];L^{mq}_{x})}\leq C_{T,\alpha} \|h\|_{L^{\infty}_{t}([0,T];L^{q}_{x})}
\eea
for some explicitly computable constant $C_{T,\alpha}$ depending only on $T$ and $\alpha$.
\end{lemma}

\section{Estimates on $K_{T}$}

We can now apply the above estimates to the $K_{T}$ term. For the $\alpha = 2$, $m=3$ case, we need $-\frac{2}{q} > -1$ and $\frac{4}{3}\leq q \leq \infty$. Thus, by Proposition \ref{FieldEst} and (\ref{OpEst}), we obtain:

\begin{prop}
For $1\leq r \leq \infty$ and $q > 2$,
\bea\label{finalKT}
\|K_{T}\|_{L^{r}_{t}([0,T];L^{3q}_{x})}\leq C_{T} \|\sigma_{-1}\|_{L^{\infty}_{t}([0,T];L^{q}_{x})}
\eea
for some explicitly computable constant $C_{T}$ depending only on $T$.
\end{prop}

To get appropriate bounds on $K_{T}$, we need to introduce some important interpolation inequalities.

\begin{lemma}\label{sigmainterpolation}
Let $1\leq r,s \leq \infty$ and suppose $\nu > 2\frac{r}{s} - 1$. Then 
\bea
\|\sigma_{-1}\|_{L^{r}_{x}} \lesssim \|p_{0}^{\nu}f\|_{L^{s}_{x}L^{1}_{p}}^{\frac{s}{r}}.
\eea
\end{lemma}

\begin{proof}
By H\"older's inequality with $\frac{1}{q} + \frac{1}{q^{'}} = 1$:

$$\int_{\mathbb{R}^{3}}{\frac{f(t,x,p)}{p_{0}(1+\hat{p}\cdot\omega)}dp}\leq \bigg( \int_{\mathbb{R}^{3}}{\frac{dp}{p_{0}^{(1+\alpha)q^{'}}(1+\hat{p}\cdot\omega)^{q^{'}}}} \bigg)^{\frac{1}{q^{'}}}\bigg( \int_{\mathbb{R}^{3}}{p_{0}^{\alpha q}f(t,x,p)^{q} dp} \bigg)^{\frac{1}{q}}$$

Call the first term on the right hand side $I$. In order to bound this term, we use the standard inequality
\bea
(1+\hat{p}\cdot\omega)^{-1} \lesssim \text{min}\{p_{0}^{2},\theta^{-2}\}
\eea
where $\theta = \angle(\frac{p}{|p|},-\omega) \in [0,\pi]$. Note that for small $\theta$, we can assume $\sin(\theta) \approx \theta$. Assuming $\alpha > 2 - \frac{1}{q}$:

\begin{equation}\label{integralbound}
\begin{split}
I & = \int_{\mathbb{R}^{3}}{\frac{dp}{p_{0}^{(1+\alpha)q^{'}}(1+\hat{p}\cdot\omega)^{q^{'}}}} \\
 & \leq \int_{\mathbb{R}^{3}}{\frac{dp}{p_{0}^{(\alpha-1)q^{'}+2}(1+\hat{p}\cdot\omega)}} \\
 & \lesssim  \lim_{P\rightarrow \infty} {\int_{0}^{P}{d|p|}\int_{0}^{2\pi}{d\phi}\bigg( \int_{0}^{p_{0}^{-1}}{\frac{p_{0}^{2}|p|^{2} \sin(\theta) d\theta}{p_{0}^{(\alpha-1)q^{'}+2}}} + \int_{p_{0}^{-1}}^{\pi}{\frac{|p|^{2} \sin(\theta) d\theta}{p_{0}^{(\alpha-1)q^{'}+2}\theta^{2}}} \bigg)} \\
 & \lesssim \lim_{P\rightarrow \infty} {\int_{0}^{P}{d|p|{\frac{1+\log(p_{0})}{p_{0}^{(\alpha-1)q^{'}}}}}} \\
 &\lesssim 1
\end{split}
\end{equation}
since $(\alpha-1)q' > 1$. Taking $L^{r}$ norm on both sides and using the conservation law $\|f\|_{L^{\infty}_{x,p}} \lesssim 1$, we obtain:
\bea
\|\sigma_{-1}\|_{L^{r}_{x}} \lesssim \|p_{0}^{\alpha q}f^{q}\|_{L^{\frac{r}{q}}_{x}L^{1}_{p}}^{\frac{1}{q}} \lesssim \|p_{0}^{\alpha q}f\|_{L^{\frac{r}{q}}_{x}L^{1}_{p}}^{\frac{1}{q}}
\eea
Hence, setting $q = \frac{r}{s}$, we finally have for $\nu > \frac{2r}{s} -1$:
\bea
\|\sigma_{-1}\|_{L^{r}_{x}} \lesssim \|p_{0}^{\nu}f\|_{L^{s}_{x}L^{1}_{p}}^{\frac{s}{r}}
\eea
This completes the proof of this inequality.
\end{proof}
 
We have the following interpolation-type inequality from Proposition 10.3 in \cite{Luk-Strain}:

\begin{prop}
Suppose $\eta$, $\rho$, and $\tau$ are real numbers such that $0< q \eta < 1$ and

$$\tau \geq \frac{\rho -\eta(N+3-3q)}{1-q\eta}$$ Then,
\bea\label{interpolation}
\|fp_{0}^{\rho}\|_{L^{\infty}_{t}([0,T]; L^{q}_{x}L^{1}_{p})} \lesssim  \|fp_{0}^{\tau}\|_{L^{\infty}_{t}([0,T]; L^{q}_{x}L^{1}_{p})}^{1-q\eta}\|fp_{0}^{N}\|_{L^{\infty}_{t}([0,T]; L^{1}_{x}L^{1}_{p})}^{\eta}
\eea
\end{prop}

Applying Lemma \ref{sigmainterpolation} and (\ref{interpolation}) to (\ref{finalKT}), we obtain for $N>3$, $1\leq r \leq \infty$ and some $\delta > 0$:

\begin{equation*}
\begin{split}
\|K_{T}\|_{L^{r}_{t}([0,T];L^{N+3}_{x})} ^{N+3} & \lesssim \|\sigma_{-1}\|_{L^{\infty}_{t}([0,T];L^{\frac{N+3}{3}}_{x})}^{N+3} \\
 & \lesssim \|p_{0}^{\frac{2N+3+\delta}{3}}f\|_{L^{1}_{x}L^{1}_{p}}^{3} \\
 & \lesssim \|p_{0}^{\frac{2N-3N\eta+3+\delta}{3(1-\eta)}}f\|_{L^{1}_{x}L^{1}_{p}}^{3-3\eta}\|p_{0}^{N}f\|_{L^{1}_{x}L^{1}_{p}}^{3\eta}
\end{split}
\end{equation*}

Setting $\eta = \frac{1-\gamma}{3}$, we obtain the needed estimate for $K_{T}$:

\begin{prop}
Given $1\leq r \leq \infty$, $\gamma \in (0,1)$, $N>3$ and $\delta > 0$,
\bea\label{KTmoment}
\|K_{T}\|_{L^{r}_{t}([0,T];L^{N+3}_{x})} ^{N+3} \lesssim \|p_{0}^{\frac{N(1+\gamma)+3+\delta}{2+\gamma}}f\|_{L^{\infty}_{t}([0,T);L^{1}_{x}L^{1}_{p})}^{2+\gamma}\|p_{0}^{N}f\|_{L^{\infty}_{t}([0,T);L^{1}_{x}L^{1}_{p})}^{1-\gamma} 
\eea
\end{prop}

\section{Estimates on $K_{S}$}

Recall Strichartz estimates for the wave operator from (4.7) in Sogge \cite{Sogge}:

\begin{thm}\label{Strichartz Estimates}(Strichartz Estimates)

Given $\lambda\in (0,1)$ and solution $u:[a,b]\times\mathbb{R}^{3}\rightarrow \mathbb{R}$ to $\square u = F$ on $[a,b] \times \mathbb{R}^{3}$ with initial data $u|_{t=a}= u(a)$ and $\partial_{t}u|_{t=a} = \partial_{t}u(a)$, there exists a constant $C_{\lambda}$ such that
$$\|u\|_{L^{\frac{2}{\lambda}}_{t}([a,b];L^{\frac{2}{1-\lambda}})} +\|u\|_{L^{\infty}_{t}([a,b];\dot{H}^{\lambda})} + \|\partial_{t} u\|_{L^{\infty}_{t}([a,b];\dot{H}^{\lambda})} 
\leq C_{\lambda} \big(  \|F\|_{L^{\frac{2}{1+\lambda}}_{t}([a,b];L^{\frac{2}{2-\lambda}})} +\|u(a)\|_{\dot{H}^{\lambda}} + \|\partial_{t} u(a)\|_{\dot{H}^{\lambda}} \big)$$

\end{thm}

Recall from the preliminary estimates (\ref{KS}) and (\ref{wave}) for $K_{S}$ that
\bea\label{KSwave}
|K_{S}| \lesssim \square^{-1}(|K|\sigma_{-1})
\eea

Using Strichartz estimates, we can prove the following:

\begin{prop}
Assume $\|\sigma_{-1}\|_{L^{\infty}_{t}([0,T];L^{2}_{x})} \lesssim 1$. Given $N > 3$, $\gamma\in(0,1)$ and $\delta>0$, we obtain the estimate
\bea\label{KSmoment}
\|K_{S}\|_{L^{1}_{t}([0,T];L^{N+3}_{x})} ^{N+3} \lesssim 1 + \|p_{0}^{\frac{N(1+\gamma)+3+\delta}{2+\gamma}}f\|_{L^{\infty}_{t}([0,T);L^{1}_{x}L^{1}_{p})}^{2+\gamma}\|p_{0}^{N}f\|_{L^{\infty}_{t}([0,T);L^{1}_{x}L^{1}_{p})}^{1-\gamma}
\eea
\end{prop}

\begin{proof}

By the above estimate (\ref{KSwave}):
\bea\label{ks1}
\|K_{S}\|_{L^{\frac{2(N+3)}{N+1}}_{t}([a,b];L^{N+3}_{x})} \leq \|\square^{-1}(|K|\sigma_{-1})\|_{L^{\frac{2(N+3)}{N+1}}_{t}([a,b];L^{N+3}_{x})}
\eea
Using the decomposition $K = K_{0} + K_{T} + K_{S}$, we obtain for some time interval $[a,b] \subset [0,T)$:
\begin{multline}\label{ks2}
\|K_{S}\|_{L^{\frac{2(N+3)}{N+1}}_{t}([a,b];L^{N+3}_{x})} \leq  \|\square^{-1}(|K_{0}|\sigma_{-1})\|_{L^{\frac{2(N+3)}{N+1}}_{t}([a,b];L^{N+3}_{x})}+ \\ \|\square^{-1}(|K_{T}|\sigma_{-1})\|_{L^{\frac{2(N+3)}{N+1}}_{t}([a,b];L^{N+3}_{x})}  +\|\square^{-1}(|K_{S}|\sigma_{-1})\|_{L^{\frac{2(N+3)}{N+1}}_{t}([a,b];L^{N+3}_{x})}
\end{multline}
Fix an interval $[a,b] \subset [0,T)$. First notice that
$$\|\square^{-1}(|K_{0}|\sigma_{-1})\|_{L^{\frac{2(N+3)}{N+1}}_{t}([a,b];L^{N+3}_{x})} \lesssim \|\square^{-1}(|K_{0}|\sigma_{-1})\|_{L^{\frac{2(N+3)}{N+1}}_{t}([0,T);L^{N+3}_{x})} $$
and similarly
$$\|\square^{-1}(|K_{T}|\sigma_{-1})\|_{L^{\frac{2(N+3)}{N+1}}_{t}([a,b];L^{N+3}_{x})} \lesssim \|\square^{-1}(|K_{T}|\sigma_{-1})\|_{L^{\frac{2(N+3)}{N+1}}_{t}([0,T);L^{N+3}_{x})} $$
Setting $\lambda = \frac{N+1}{N+3}$, we obtain by the Strichartz estimates for the wave operator:
$$\|\square^{-1}(|K_{0}|\sigma_{-1})\|_{L^{\frac{2(N+3)}{N+1}}_{t}([0,T);L^{N+3}_{x})} \lesssim C_{\frac{N+1}{N+3}}\||K_{0}|\sigma_{-1}\|_{L^{\frac{(N+3)}{N+2}}_{t}([0,T);L^{\frac{2(N+3)}{N+5}}_{x})}$$
since we have trivial initial data by (\ref{wave}). Applying the same argument to the $K_{T}$ term and by (\ref{ks2}) we obtain:
\begin{multline}\label{ks3}
\|K_{S}\|_{L^{\frac{2(N+3)}{N+1}}_{t}([a,b];L^{N+3}_{x})}\leq C_{\frac{N+1}{N+3}}\||K_{0}|\sigma_{-1}\|_{L^{\frac{(N+3)}{N+2}}_{t}([0,T);L^{\frac{2(N+3)}{N+5}}_{x})} + C_{\frac{N+1}{N+3}}\||K_{T}|\sigma_{-1}\|_{L^{\frac{(N+3)}{N+2}}_{t}([0,T);L^{\frac{2(N+3)}{N+5}}_{x})} \\  +\|\square^{-1}(|K_{S}|\sigma_{-1})\|_{L^{\frac{2(N+3)}{N+1}}_{t}([a,b];L^{N+3}_{x})}
\end{multline}
Applying H\"older's inequality with $\frac{1}{2}+\frac{1}{N+3} = \frac{N+5}{2(N+3)}$:
\begin{multline}\label{ks4}
\|K_{S}\|_{L^{\frac{2(N+3)}{N+1}}_{t}([a,b];L^{N+3}_{x})}\leq C_{\frac{N+1}{N+3}}\|K_{0}\|_{L^{\frac{2(N+3)}{N+1}}_{t}([0,T);L^{N+3}_{x})}\|\sigma_{-1}\|_{L^{2}_{t}([0,T);L^{2}_{x})} \\ + C_{\frac{N+1}{N+3}}\|K_{T}\|_{L^{\frac{2(N+3)}{N+1}}_{t}([0,T);L^{N+3}_{x})}\|\sigma_{-1}\|_{L^{2}_{t}([0,T);L^{2}_{x})} +\|\square^{-1}(|K_{S}|\sigma_{-1})\|_{L^{\frac{2(N+3)}{N+1}}_{t}([a,b];L^{N+3}_{x})}
\end{multline}
Note that we can bound $ C_{\frac{N+1}{N+3}}\|K_{0}\|_{L^{\frac{2(N+3)}{N+1}}_{t}([0,T);L^{N+3}_{x})}\|\sigma_{-1}\|_{L^{2}_{t}([0,T);L^{2}_{x})}$ by a constant since $K_{0}$ depends only on initial data and we have assumed that $\|\sigma_{-1}\|_{L^{2}_{t}([0,T);L^{2}_{x})} \lesssim_{T} \|\sigma_{-1}\|_{L^{\infty}_{t}([0,T);L^{2}_{x})} \lesssim 1$. Finally, using the estimate (\ref{KTmoment}) on $K_{T}$ from the previous section,
\begin{multline}\label{ks5}
\|K_{S}\|_{L^{\frac{2(N+3)}{N+1}}_{t}([a,b];L^{N+3}_{x})} \\  
 \leq (data) + C \|p_{0}^{\frac{N(1+\gamma)+3+\delta}{2+\gamma}}f\|_{L^{\infty}_{t}([0,T);L^{1}_{x}L^{1}_{p})}^{\frac{2+\gamma}{N+3}}\|p_{0}^{N}f\|_{L^{\infty}_{t}([0,T);L^{1}_{x}L^{1}_{p})}^{\frac{1-\gamma}{N+3}}\|\sigma_{-1}\|_{L^{2}_{t}([0,T];L^{2}_{x})} \\
 + \|\square^{-1}(|K_{S}|\sigma_{-1})\|_{L^{\frac{2(N+3)}{N+1}}_{t}([a,b];L^{N+3}_{x})}
\end{multline}
Similarly, we can now apply Strichartz estimates and H\"older's inequality to the $K_{S}$ term. Note that we kept the time interval on the $K_{S}$ term as $[a,b]$. Setting $ u = \square^{-1}(|K_{S}|\sigma_{-1})$,
\begin{multline}\label{ks6}
\|\square^{-1}(|K_{S}|\sigma_{-1})\|_{L^{\frac{2(N+3)}{N+1}}_{t}([a,b];L^{N+3}_{x})} \leq
C_{\frac{N+1}{N+3}}  ( \|K_{S}\|_{L^{\frac{2(N+3)}{N+1}}_{t}([a,b];L^{N+3})}\|\sigma_{-1}\|_{L^{2}_{t}([a,b];L^{2}_{x})}  \\ +\|u(a)\|_{\dot{H}^{\frac{N+1}{N+3}}}  + \|\partial_{t} u(a)\|_{\dot{H}^{\frac{N+1}{N+3}}})
\end{multline}
Next, we can choose a partition $0 = T_{0} < T_{1} < T_{2} < \ldots < T_{k-1} < T_{k} = T$ of $[0,T]$ such that

$$ \|\sigma_{-1}\|_{L^{2}_{t}([T_{i},T_{i+1}];L^{2}_{x})} \leq \frac{1}{2C_{\frac{N+1}{N+3}}} \ for \ i\in\{0, 1, \ldots, k-1\}$$
\\
due to the assumption $\|\sigma_{-1}\|_{L^{\infty}_{t}([0,T];L^{2}_{x})} \leq \tilde{C}$ for some $\tilde{C}$. (For example, we can choose our partition so that $(T_{i}-T_{i-1})^{\frac{1}{2}} \leq \frac{1}{2\tilde{C}2C_{\frac{N+1}{N+3}}}$ for $i = 1,2,\ldots, k$.) Using $(\ref{ks5})$ and $(\ref{ks6})$:

\begin{multline}\label{ks7}
\|K_{S}\|_{L^{\frac{2(N+3)}{N+1}}_{t}([T_{i},T_{i+1}];L^{N+3}_{x})} \\  
 \leq 2(data)_{i} + 2 C \|p_{0}^{\frac{N(1+\gamma)+3+\delta}{2+\gamma}}f\|_{L^{\infty}_{t}([0,T);L^{1}_{x}L^{1}_{p})}^{\frac{2+\gamma}{N+3}}\|p_{0}^{N}f\|_{L^{\infty}_{t}([0,T);L^{1}_{x}L^{1}_{p})}^{\frac{1-\gamma}{N+3}}\\
 + 2 C_{\frac{N+1}{N+3}} \big( \|u(T_{i})\|_{\dot{H}^{\frac{N+1}{N+3}}} + \|\partial_{t} u(T_{i})\|_{\dot{H}^{\frac{N+1}{N+3}}} \big)
\end{multline}
Thus, by H\"older's inequality, our choice of partition, and then (\ref{ks7}):

\begin{multline}
\||K_{S}|\sigma_{-1}\|_{L^{\frac{N+3}{N+2}}_{t}([T_{i},T_{i+1}];L^{\frac{2(N+3)}{N+5}}_{x})} \leq \frac{1}{2C_{\frac{N+1}{N+3}}}\|K_{S}\|_{L^{\frac{2(N+3)}{N+1}}_{t}([T_{i},T_{i+1}];L^{N+3}_{x})}\\  
 \leq \frac{1}{C_{\frac{N+1}{N+3}}}(data)_{i} + \frac{1}{C_{\frac{N+1}{N+3}}} C \|p_{0}^{\frac{N(1+\gamma)+3+\delta}{2+\gamma}}f\|_{L^{\infty}_{t}([0,T);L^{1}_{x}L^{1}_{p})}^{\frac{2+\gamma}{N+3}}\|p_{0}^{N}f\|_{L^{\infty}_{t}([0,T);L^{1}_{x}L^{1}_{p})}^{\frac{1-\gamma}{N+3}}\\
 + \|u(T_{i})\|_{\dot{H}^{\frac{N+1}{N+3}}} + \|\partial_{t} u(T_{i})\|_{\dot{H}^{\frac{N+1}{N+3}}}
\end{multline}
Using the Strichartz estimates again for $[T_{i-1},T_{i}]$, we obtain:
\begin{multline*}
\|u(T_{i})\|_{\dot{H}^{\frac{N+1}{N+3}}} + \|\partial_{t} u(T_{i})\|_{\dot{H}^{\frac{N+1}{N+3}}} \\
\leq C_{\frac{N+1}{N+3}} \bigg( \|u(T_{i-1})\|_{\dot{H}^{\frac{N+1}{N+3}}} + \|\partial_{t} u(T_{i-1})\|_{\dot{H}^{\frac{N+1}{N+3}}} +  \||K_{S}|\sigma_{-1}\|_{L^{\frac{N+3}{N+2}}_{t}([T_{i-1},T_{i}];L^{\frac{2(N+3)}{N+5}}_{x})} \bigg) \\
\leq (data)_{i} + C \|p_{0}^{\frac{N(1+\gamma)+3+\delta}{2-\gamma}}f\|_{L^{\infty}_{t}([0,T);L^{1}_{x}L^{1}_{p})}^{\frac{2+\gamma}{N+3}}\|p_{0}^{N}f\|_{L^{\infty}_{t}([0,T);L^{1}_{x}L^{1}_{p})}^{\frac{1-\gamma}{N+3}}\\
 + 2 C_{\frac{N+1}{N+3}} \big( \|u(T_{i-1})\|_{\dot{H}^{\frac{N+1}{N+3}}} + \|\partial_{t} u(T_{i-1})\|_{\dot{H}^{\frac{N+1}{N+3}}} \big)
\end{multline*}
Thus, since $u(0) = \partial_{t} u(0) = 0$, we do an iteration of the above to get the following estimate:
\begin{multline}
\|u(T_{j+1})\|_{\dot{H}^{\frac{N+1}{N+3}}} + \|\partial_{t} u(T_{j+1})\|_{\dot{H}^{\frac{N+1}{N+3}}} \\
\leq \sum_{i=0}^{j}{(2 C_{\frac{N+1}{N+3}})^{j-i}\bigg( (data)_{i} + C \|p_{0}^{\frac{N(1+\gamma)+3+\delta}{2-\gamma}}f\|_{L^{\infty}_{t}([0,T);L^{1}_{x}L^{1}_{p})}^{\frac{2+\gamma}{N+3}}\|p_{0}^{N}f\|_{L^{\infty}_{t}([0,T);L^{1}_{x}L^{1}_{p})}^{\frac{1-\gamma}{N+3}} \bigg)}
\end{multline}
Plugging this estimate into (\ref{ks7}) and using the triangle inequality to sum over the entire partition,

$$ \|K_{S}\|_{L^{\frac{2(N+3)}{N+1}}_{t}([T_{i},T_{i+1}];L^{N+3}_{x})} \lesssim \sum\limits_{i=0}^{k-1}\|K_{S}\|_{L^{\frac{2(N+3)}{N+1}}_{t}([T_{i},T_{i+1}];L^{N+3}_{x})} \lesssim 1 + \|p_{0}^{\frac{N(1+\gamma)+3+\delta}{2+\gamma}}f\|_{L^{\infty}_{t}([0,T);L^{1}_{x}L^{1}_{p})}^{\frac{2+\gamma}{N+3}}\|p_{0}^{N}f\|_{L^{\infty}_{t}([0,T);L^{1}_{x}L^{1}_{p})}^{\frac{1-\gamma}{N+3}} $$
which implies the estimate (\ref{KSmoment}) after an application of H\"{o}lder's inequality in the time variable.
\end{proof}

\section{Bounds on Higher Moments}

We have the standard moment estimate recalled from Proposition 7.3 in \cite{Luk-Strain}:

\begin{prop}\label{lukstrainmoment}
Given $N>0$, we have the uniform estimate

$$ \|p_{0}^{N}f\|_{L^{\infty}_{t}([0,T);L^{1}_{x}L^{1}_{p})} \lesssim   \|p_{0}^{N}f_{0}\|_{L^{\infty}_{t}([0,T);L^{1}_{x}L^{1}_{p})} + \|K\|_{L^{1}_{t}([0,T);L^{N+3}_{x})}^{N+3} $$
\end{prop}

It suffices to bound moments $ \|p_{0}^{N}f\|_{L^{\infty}_{t}([0,T);L^{1}_{x}L^{1}_{p})} $ for all $N >0$ due to the following sharpened estimate from \cite{Luk-Strain}:

\begin{prop}\label{Jacobian}
Over any spatial characteristic curve $X(s;t,x,p)$ we have the bound:
\begin{multline}
\sup\limits_{(t,x,p)\in\mathbb{R_{+}}\times\mathbb{R}_{x}^{3}\times\mathbb{R}_{p}^{3}} \int_{0}^{T_{*}} \big( |E(s,X(s;t,x,p)| + |B(s,X(s;t,x,p)|  \big) ds \lesssim 1 + \|K\sigma_{-1}\|_{L^{\infty}_{t}([0,T];L^{2+\lambda}_{x}L^{1}_{p})} \\ + \|\sigma_{-1}\|_{L^{\infty}_{t}([0,T];L^{3+\tilde{\lambda}}_{x}L^{1}_{p})}
\end{multline}
for any $\lambda, \tilde{\lambda} > 0$.
\end{prop}

\begin{proof}
We can rewrite the bounds (\ref{KT}) and (\ref{KS})  in the form:
\bea\label{LSdecomp1}
|K_{T}|(t,x) \lesssim \int_{C_{t,x}}\frac{\sigma_{-1}(s,y)}{(t-s)^{2}} d\sigma
\eea
\bea\label{LSdecomp2}
|K_{S}|(t,x) \lesssim \int_{C_{t,x}}\frac{(|K|\sigma_{-1})(s,y)}{t-s} d\sigma
\eea
where the integral over the cone $C_{t,x}$ is given (\ref{coneintegral}). Using (\ref{LSdecomp1}) and (\ref{LSdecomp2}), we can bound our integral over the characteristic $X(s;t,x,p)$ by:
\begin{multline}
sup \int_{0}^{T_{*}} \big( |E(s,X(s;t,x,p)| + |B(s,X(s;t,x,p)|  \big) ds \\ \lesssim 1 + \int_{0}^{T_{*}} \int_{C_{s,X(s)}}\frac{\sigma_{-1}(\tilde{s},y)}{(s-\tilde{s})^{2}} d\sigma ds +\int_{0}^{T_{*}} \int_{C_{s,X(s)}}\frac{(|K|\sigma_{-1})(\tilde{s},y)}{(s-\tilde{s})} d\sigma ds
\end{multline}
where $d\sigma = d\sigma(\tilde{s},y) = (s-\tilde{s})^{2}\sin(\theta) d\tilde{s} d\phi d\theta $ and $X(s) = X(s;t,x,p)$. The integral terms on the right hand side have the general form:
\bea\label{above}
I_{i}(t,x) \eqdef \int_{0}^{t} \int_{C_{s,X(s)}}\frac{g_{i}}{(s-\tilde{s})^{i}} d\sigma ds \ \ (i = 1,2),
\eea
where $g_{1} = |K|\sigma_{-1}$ and $g_{2} = \sigma_{-1}$. By a change of variables after writing (\ref{above}) expanded as (\ref{coneintegral}), we obtain:

\bea
I_{i}(t,x) =  \int_{0}^{t} \int_{\tilde{s}}^{t} \int_{0}^{2\pi}\int_{0}^{\pi} (s-\tilde{s})^{2-i}\sin(\theta) g_{i}(s, X(s)+(s-\tilde{s})\omega) d\theta d\phi ds d\tilde{s} \eqdef \int_{0}^{t} J_{i}(\tilde{s},X(\tilde{s})) d\tilde{s},
\eea
where again we have adopted the convention $X(s) = X(s;t,x,p)$.

Following Pallard \cite{Pallard}, we define the diffeomorphism $\pi \eqdef X(s) + (s-\tilde{s})\omega$. This change of variables has Jacobian $J_{\pi} = (X'(s)\cdot \omega + 1)(s-\tilde{s})^{2}\sin(\theta) \neq 0$ on $\theta\in(0,\pi)$ (since $|X'(s)| \leq |\hat{V}(s)| < 1$ and hence $X'(s)\cdot \omega + 1 > 0$). First using H\"{o}lder's inequality for H\"older exponents $q,q'$ such that $\frac{1}{q} + \frac{1}{q'} = 1$:
\begin{multline}\label{Jpi}
J_{i}(\tilde{s},X(\tilde{s})) \leq \bigg(\int_{\tilde{s}}^{t} \int_{0}^{2\pi}\int_{0}^{\pi} \frac{(s-\tilde{s})^{(2-i)q'}\sin^{q'}(\theta)}{J_{\pi}^{\frac{q'}{q}}} d\theta d\phi ds \bigg)^{\frac{1}{q'}} \\ \times\bigg(\int_{\tilde{s}}^{t} \int_{0}^{2\pi}\int_{0}^{\pi} g_{i}(s, X(s)+(s-\tilde{s})\omega)^{q} J_{\pi} d\theta d\phi ds \bigg)^{\frac{1}{q}}
\end{multline}
Next, using the change of variables described by the diffeomorphism $\pi$ in the second integral on the right hand side of (\ref{Jpi}):

\bea
J_{i}(\tilde{s},X(\tilde{s})) \leq \bigg(\int_{\tilde{s}}^{t} \int_{0}^{2\pi}\int_{0}^{\pi} \frac{(s-\tilde{s})^{(2-i)q'}\sin^{q'}(\theta)}{J_{\pi}^{\frac{q'}{q}}} d\theta d\phi ds \bigg)^{\frac{1}{q'}} \|g_{i}(\tilde{s})\|_{L^{q}(\mathbb{R}^{3})}
\eea

Finally, plugging in the expression for $J_{\pi}$ into the remaining integral on the right hand side, we see that it is bounded for certain choices of $q$. To see this, choose coordinates $(\theta, \phi)$ such that $X'\cdot \omega = -|X'||\omega|\cos(\theta) \geq -\cos(\theta) $. Then, using this coordinate system:

\bea
\int_{\tilde{s}}^{t} \int_{0}^{2\pi}\int_{0}^{\pi} \frac{(s-\tilde{s})^{(2-i)q'}\sin^{q'}(\theta)}{J_{\pi}^{\frac{q'}{q}}} d\theta d\phi ds \lesssim \int_{\tilde{s}}^{t} \int_{0}^{2\pi}\int_{0}^{\pi} \frac{(s-\tilde{s})^{(2-i)q'-\frac{2q'}{q}}\sin^{q'-\frac{q'}{q}}(\theta)}{(1-\cos(\theta))^{\frac{q'}{q}}} d\theta d\phi ds
\eea
Now, note that $\frac{1}{1-\cos(\theta)} = \frac{1}{1-\sqrt{1-\sin^{2}(\theta)}} = \frac{1+\sqrt{1-\sin^{2}(\theta)}}{\sin^{2}(\theta)} \lesssim \frac{1}{\sin^{2}(\theta)}$ since $1+\sqrt{1-\sin^{2}(\theta)} \leq 2$. Plugging this into the above, we obtain:

\bea
\int_{\tilde{s}}^{t} \int_{0}^{2\pi}\int_{0}^{\pi} \frac{(s-\tilde{s})^{(2-i)q'}\sin^{q'}(\theta)}{J_{\pi}^{\frac{q'}{q}}} d\theta d\phi ds = \int_{\tilde{s}}^{t} \int_{0}^{2\pi}\int_{0}^{\pi} (s-\tilde{s})^{(2-i)q'-\frac{2q'}{q}}\sin^{q'- \frac{3q'}{q}}(\theta) d\theta d\phi ds
\eea
The integral over $\theta$ remains bounded when $q'-\frac{3q'}{q} > -1$ and $(2-i)q' -\frac{2q'}{q} > -1$, i.e. when $q>2$ and $q > \frac{3}{3-i}$ for $i=1,2$.
\end{proof}
From the above estimate, we can use H\"older's inequality and Lemma \ref{sigmainterpolation} to obtain for any $\lambda',\tilde{\lambda},\hat{\lambda} > 0$:
\begin{multline}\label{momentboundfieldchar}
\|K\sigma_{-1}\|_{L^{\infty}_{t}([0,T];L^{2+\lambda}_{x}L^{1}_{p})} + \|\sigma_{-1}\|_{L^{\infty}_{t}([0,T];L^{3+\tilde{\lambda}}_{x}L^{1}_{p})} \\
\lesssim \|K\|_{L^{\infty}_{t}([0,T];L^{3+\lambda'}_{x})}\|\sigma_{-1}\|_{L^{\infty}_{t}([0,T];L^{3}_{x}L^{1}_{p})} + \|\sigma_{-1}\|_{L^{\infty}_{t}([0,T];L^{3+\tilde{\lambda}}_{x}L^{1}_{p})} \\
\lesssim \|K\|_{L^{\infty}_{t}([0,T];L^{3+\lambda'}_{x})}\|p_{0}^{5+2\hat{\lambda}}f\|_{L^{\infty}_{t}([0,T];L^{1}_{x}L^{1}_{p})}^{\frac{1}{3+\hat{\lambda}}} + \|p_{0}^{5+2\tilde{\lambda}}f\|_{L^{\infty}_{t}([0,T];L^{1}_{x}L^{1}_{p})}^{\frac{1}{3+\tilde{\lambda}}}
\end{multline}
By interpolation and the conservation law (\ref{conserve1}), there exists some $\theta \in (0,1)$ such that
$$\|K\|_{L^{\infty}_{t}([0,T];L^{3+\lambda'}_{x})} \lesssim \|K\|_{L^{\infty}_{t}([0,T];L^{6+\lambda'}_{x})}^{\theta}\|K\|_{L^{\infty}_{t}([0,T];L^{2}_{x})}^{1-\theta} \lesssim \|K\|_{L^{\infty}_{t}([0,T];L^{6+\lambda'}_{x})}^{\theta}.$$
By the estimates (\ref{KTmoment}) and (\ref{KSmoment}), we can further bound this by:
\bea\label{furtherbound}
\|K\|_{L^{\infty}_{t}([0,T];L^{6+\lambda'}_{x})} \lesssim 1 + \|p_{0}^{\frac{(3+\lambda')(1+\gamma)+3+\delta}{2+\gamma}}f\|_{L^{\infty}_{t}([0,T);L^{1}_{x}L^{1}_{p})}^{2+\gamma}\|p_{0}^{3+\lambda'}f\|_{L^{\infty}_{t}([0,T);L^{1}_{x}L^{1}_{p})}^{1-\gamma}
\eea
for any $\delta > 0$ and $\gamma \in (0,1)$. Choosing $\delta < \lambda'$, we obtain that
$$ \frac{(3+\lambda')(1+\gamma)+3+\delta}{2+\gamma} < 3 + \lambda',$$ and hence by (\ref{furtherbound}):
$$\|K\|_{L^{\infty}_{t}([0,T];L^{6+\lambda'}_{x})} \lesssim 1 + \|p_{0}^{3+\lambda'}f\|_{L^{\infty}_{t}([0,T);L^{1}_{x}L^{1}_{p})}^{3}. $$ Putting these bounds together, we obtain:
\begin{multline}\label{totalbounds}
\|K\sigma_{-1}\|_{L^{\infty}_{t}([0,T];L^{2+\lambda}_{x}L^{1}_{p})} + \|\sigma_{-1}\|_{L^{\infty}_{t}([0,T];L^{3+\tilde{\lambda}}_{x}L^{1}_{p})}
\lesssim (1 +\|p_{0}^{3+\lambda'}f\|_{L^{\infty}_{t}([0,T);L^{1}_{x}L^{1}_{p})}^{3})\|p_{0}^{5+2\hat{\lambda}}f\|_{L^{\infty}_{t}([0,T];L^{1}_{x}L^{1}_{p})}^{\frac{1}{3+\hat{\lambda}}} \\ + \|p_{0}^{5+2\tilde{\lambda}}f\|_{L^{\infty}_{t}([0,T];L^{1}_{x}L^{1}_{p})}^{\frac{1}{3+\tilde{\lambda}}}
\end{multline}

Thus, in order to satisfy the known continuation criteria stated in Theorem \ref{LSorigin}, we simply need to bound $\|p_{0}^{5+\lambda}f\|_{L^{\infty}_{t}([0,T];L^{1}_{x}L^{1}_{p})} \lesssim 1$ for some $\lambda >0$. To this end, we can use the estimates on $K$ to prove that:

\begin{prop}\label{finalest}
Consider initial data $f_{0}$ such that $\|p_{0}^{N}f_{0}\|_{L^{1}_{x}L^{1}_{p}} \lesssim 1$ and suppose we have the bound $\|\sigma_{-1}\|_{L^{\infty}_{t}([0,T];L^{2}_{x})}  + \|p_{0}^{M}f\|_{L^{\infty}_{t}([0,T];L^{1}_{x}L^{1}_{p})} \lesssim 1$ where $M > \frac{N+3}{2}$ for some $N>3$. Then
$$ \|p_{0}^{N}f\|_{L^{\infty}_{t}([0,T];L^{1}_{x}L^{1}_{p})} \lesssim 1 $$ and
$$ \|K\|_{L^{\infty}_{t}([0,T];L^{N+3}_{x})}\lesssim 1$$.
\end{prop}

\begin{proof}

By the estimates on $K_{T}$ and $K_{S}$ given by (\ref{KTmoment}) and (\ref{KSmoment}) respectively and Proposition \ref{lukstrainmoment}, we obtain for some $\gamma \in (0,1)$:
\bea\label{finalest1}
\|p_{0}^{N}f\|_{L^{\infty}_{t}([0,T];L^{1}_{x}L^{1}_{p})} \lesssim 1 + \|p_{0}^{\frac{N(1+\gamma)+3+\delta}{2+\gamma}}f\|_{L^{\infty}_{t}([0,T);L^{1}_{x}L^{1}_{p})}^{2+\gamma}\|p_{0}^{N}f\|_{L^{\infty}_{t}([0,T);L^{1}_{x}L^{1}_{p})}^{1-\gamma}
\eea

Choose appropriate $0 < \gamma < 1$ and $\delta > 0$ such that $\frac{N(1+\gamma)+3+\delta}{2+\gamma} = M$ and let the implicit constant in (\ref{finalest1}) be denoted by $C > 0$. (Suppose $M= \frac{N+3+\epsilon}{2}$. Then set $\delta = \epsilon + \Big(\frac{N+3+\epsilon}{2}-N\Big)\gamma$. For $\gamma \in (0,1)$ sufficiently small, $\delta > 0$.) Thus, since $\|p_{0}^{M}f\|_{L^{\infty}_{t}([0,T];L^{1}_{x}L^{1}_{p})} \leq B$ for some constant $B > 0$ and by Young's inequality:
\bea\label{youngsineq}
\|p_{0}^{N}f\|_{L^{\infty}_{t}([0,T];L^{1}_{x}L^{1}_{p})} \leq C + C B^{\beta}\|p_{0}^{N}f\|_{L^{\infty}_{t}([0,T);L^{1}_{x}L^{1}_{p})}^{1-\gamma} \leq C + \gamma C^{\frac{1}{\gamma}} B^{\frac{2+\gamma}{\gamma}} + (1-\gamma)\|p_{0}^{N}f\|_{L^{\infty}_{t}([0,T];L^{1}_{x}L^{1}_{p})}
\eea
Thus for some constant $\tilde{C}$,

\bea\label{finalest2}
\|p_{0}^{N}f\|_{L^{\infty}_{t}([0,T];L^{1}_{x}L^{1}_{p})} \leq \frac{1}{\gamma}(C + \gamma \tilde{C} B^{\frac{2+\gamma}{\gamma}}) \lesssim 1
\eea

Finally, plugging (\ref{finalest2}) into (\ref{KTmoment}) and (\ref{KSmoment}), we obtain that $ \|K\|_{L^{\infty}_{t}([0,T];L^{N+3}_{x})}\lesssim 1$.
\end{proof}

\begin{thm}
Suppose $\|p_{0}^{\tilde{N}}f_{0}\|_{L^{1}_{x}L^{1}_{p}} \lesssim 1$ for some $\tilde{N} > 5$. Let $M > 3$. Then $\|p_{0}^{M}f\|_{L^{\infty}_{t}([0,T];L^{1}_{x}L^{1}_{p})} \lesssim 1$ is a continuation criteria for the Vlasov-Maxwell system without compact support.
\end{thm}

\begin{proof}
First, if $M>5$, then by the comment under (\ref{totalbounds}), we are done. (Note that this is also a known continuation criteria found in \cite{Luk-Strain}.)
If $3<M<5$, note that by Lemma \ref{sigmainterpolation} $$\|\sigma_{-1}\|_{L^{\infty}_{t}([0,T];L^{2}_{x})}  + \|p_{0}^{M}f\|_{L^{\infty}_{t}([0,T];L^{1}_{x}L^{1}_{p})} \lesssim \|p_{0}^{M}f\|_{L^{\infty}_{t}([0,T];L^{1}_{x}L^{1}_{p})} \lesssim 1.$$
Suppose $M>3$ and $\|p_{0}^{M}f\|_{L^{\infty}_{t}([0,T];L^{1}_{x}L^{1}_{p})} \lesssim 1$. Since $\|p_{0}^{\tilde{N}}f_{0}\|_{L^{1}_{x}L^{1}_{p}} \lesssim 1$ for some $\tilde{N} > 5$, it follows that $\|p_{0}^{N}f_{0}\|_{L^{1}_{x}L^{1}_{p}} \lesssim 1$ for all $N < 5$. Then, by Proposition \ref{finalest}, we obtain that $\|p_{0}^{N}f\|_{L^{\infty}_{t}([0,T];L^{1}_{x}L^{1}_{p})} \lesssim 1$ for $N=2M-3-\delta$ for $\delta > 0$ as long as $3 < 2M-3-\delta < 5$. Note that if $M>3$, then $2M-3 = M + M-3 > M$. Hence setting $\delta = \frac{M-3}{2}$, we obtain that $N=2M-3-\delta = 2M-3-\frac{M-3}{2} = M + \frac{M-3}{2} > M > 3$.

Let $M = M_{0}$ and suppose, as above, that $$\|p_{0}^{M}f\|_{L^{\infty}_{t}([0,T];L^{1}_{x}L^{1}_{p})} = \|p_{0}^{M_{0}}f\|_{L^{\infty}_{t}([0,T];L^{1}_{x}L^{1}_{p})} \lesssim 1.$$
Then, if $M_{1} = M_{0} + \frac{M_{0}-3}{2} < 5$, we know by the above that
$$\|p_{0}^{M_{1}}f\|_{L^{\infty}_{t}([0,T];L^{1}_{x}L^{1}_{p})} \lesssim 1.$$
Define the sequence $M_{i}$ in this manner: let $M_{i+1} = M_{i} + \frac{M_{i}-3}{2}$. Notice that since $M_{0} > 3$, by the earlier argument, we obtain that $M_{1} >3$. By induction, we obtain that $M_{k} > 3$ for all $k\in \mathbb{N}$.

Since $M=M_{0} >3$, there exists an $\epsilon > 0$ such that $ M_{0} = 3+\epsilon$. We now claim that $M_{k} > M_{0} + \frac{k\epsilon}{2}$. Indeed, this is true in the case of $M_{0}$. Suppose it holds for $k= n$. Then, since $M_{n}-3 > M_{0} -3 + \frac{n\epsilon}{2} = \frac{\epsilon}{2} + \frac{n\epsilon}{4} > \frac{\epsilon}{2}$, it follows that $M_{n+1} = M_{n}+\frac{M_{n}-3}{2} > M_{0} + \frac{n\epsilon}{2} + \frac{\epsilon}{2} = M_{0} + \frac{(n+1)\epsilon}{2}$.

Thus, as $n$ tends to infinity, we know that $M_{n}$ tends to infinity. Thus, there exists some $m \in \mathbb{N}$ such that $3< M_{m} < 5$ but $M_{m+1}>5$. Under our assumption that $\|p_{0}^{M_{0}}f\|_{L^{\infty}_{t}([0,T];L^{1}_{x}L^{1}_{p})} \lesssim 1$, we can iterate the argument above to obtain that $\|p_{0}^{M_{n}}f\|_{L^{\infty}_{t}([0,T];L^{1}_{x}L^{1}_{p})} \lesssim 1$ for all positive integers $n \leq m$. Finally, choose some $\delta > 0$ such that $5 < 2M_{m}-3-\delta < \tilde{N}$.  (This is certainly possible since choosing $\delta = \frac{M_{i}-3}{2}$, we obtain by our choice of $m$ that $M_{m+1}= 2M_{m}-3-\delta > 5$. On the other hand, if $M_{m+1} > \tilde{N} > 5$, we simply choose a large delta such that $2M-3-\delta$ is still greater than $5$ but is less than $\tilde{N}$.) Let us set $\tilde{M} = 2M_{m}-3-\delta$. Since $\tilde{M} < \tilde{N}$, we know that $\|p_{0}^{\tilde{M}}f_{0}\|_{L^{1}_{x}L^{1}_{p}} \lesssim 1$. By Proposition \ref{finalest}, we obtain that $\|p_{0}^{\tilde{M}}\|_{L^{1}_{x}L^{1}_{p}} \lesssim 1$. Since $\tilde{M}>5$, by the comment under (\ref{totalbounds}), we are done.
%Let $M_{0} = M$ and $M_{1}= 2M_{0}-3$. Then, for any $\delta > 0$, we obtain the bound $\|p_{0}^{M_{1}-\delta}f\|_{L^{\infty}_{t}([0,T];L^{1}_{x}L^{1}_{p})} \lesssim 1$. By an iteration of this procedure, we get that $\|p_{0}^{M_{i+1}-\delta}f\|_{L^{\infty}_{t}([0,T];L^{1}_{x}L^{1}_{p})} \lesssim 1$ for any $\delta >0$ where $M_{i+1} = 2M_{i} -3$. Now, suppose $M = 3 + \epsilon$ for from $\epsilon >0$. Then, by induction, for any $\delta >0$, we know that $M_{k} = 3 + 2^{k}\epsilon$. (Suppose $M_{i} = 3 + 2^{i}\epsilon$. Then, $M_{i+1} = 2M_{i} - 3 = 2 (3 + 2^{i}\epsilon) - 3 = 3 + 2^{i+1}\epsilon$.)Suppose  $\|p_{0}^{M}f\|_{L^{\infty}_{t}([0,T];L^{1}_{x}L^{1}_{p})} \lesssim 1$ where $M = 3+\epsilon$ for some fixed $\epsilon >0$. Given $N>0$, there exists $k\in\mathbb{Z}_{\geq 0}$ such that $N < 3+2^{k}\epsilon$. In particular, there exists a $\delta > 0$ such that $N = 3+2^{k}\epsilon - \delta$. Hence, by the above argument, $\|p_{0}^{N}f\|_{L^{\infty}_{t}([0,T];L^{1}_{x}L^{1}_{p})} \lesssim 1$. This bound holds for all $N>0$, which as mentioned in the comment above Proposition \ref{Jacobian}, is a continuation criteria for our system. %Hence, by the above argument, $\|p_{0}^{N}f\|_{L^{\infty}_{t}([0,T];L^{1}_{x}L^{1}_{p})} \lesssim 1$ and by Proposition\ref{finalest}, $\|K\|_{L^{\infty}_{t}([0,T];L^{N+3}_{x})} \lesssim 1$. Thus, in particular, the moment norms on the right hand side of (\ref{momentboundfieldchar}) are bounded.%
\end{proof}

\section{Another Field Decomposition}

In this section, we recall the decomposition found in Luk-Strain \cite{L-S} and bound each piece in the form of the operator $W_{2}$ or the inverse d-Alembertain $\square^{-1}$. Note that from this point in this paper, we define 
$$|K| \eqdef |E| + |B|$$
Then, we have that $|K| \leq |K_{0}| + |K_{T}| + |K_{S,1}| + |K_{S,2}|$ where

\begin{prop}\label{LSplanedecomp}
We have the following estimates:
\bea\label{ktls}
|K_{T}| = |E_{T}| + |B_{T}| \lesssim W_{2}(\sigma_{-1})
\eea

\bea\label{ks1ls}
|K_{S,1}| = |E_{S,1}| + |B_{S,1}| \lesssim \square^{-1} (|K| \Phi_{-1})
\eea

\bea\label{ks2ls}
|K_{S,2}| = |E_{S,2}| + |B_{S,2}| \lesssim (W_{2}(\sigma_{-1}^{2}))^{\frac{1}{2}}
\eea

where
\bea\label{phimoment}
\Phi_{-1}(t,x) \eqdef \max\limits_{|\omega|=1}\int_{\mathbb{R}^{3}} \frac{f(t,x,p)dp}{p_{0}(1+\hat{p}\cdot\omega)^{\frac{1}{2}}}.
\eea
\end{prop}

\begin{proof}
Following the decomposition of \cite{L-S} we have that 
\bea
(|E_{T}| + |B_{T}|)(t,x) \lesssim \int_{C_{t,x}}\int_{\mathbb{R}^{3}} \frac{f(s,x+(t-s)\omega,p)}{(t-s)^{2}p_{0}(1+\hat{p}\cdot \omega)}dp d\sigma(\omega)
\eea

Using the change of variable $t-s \rightarrow s$ and writing the integral over the cone $C_{t,x}$ as an integral over spheres of radius $s$, we obtain:
\bea\label{changekt}
(|E_{T}| + |B_{T}|)(t,x) \lesssim \int_{C_{t,x}}\int_{\mathbb{R}^{3}} \frac{f(t-s,x+s\omega,p)}{s^{2}p_{0}(1+\hat{p}\cdot \omega)}dp d\sigma(\omega) \leq \int_{0}^{t}\fint_{\mathbb{S}^{2}} \sigma_{-1}(t-s,x+s\omega) d\sigma(\omega)
\eea
which is of the form $W_{2}(\sigma_{-1})$.

Next, by Proposition 3.4 in \cite{L-S}:
\bea
(|E_{S,1}| + |B_{S,1}|)(t,x) \lesssim  \int_{C_{t,x}}\int_{\mathbb{R}^{3}} \frac{|B|f(s,x+(t-s)\omega,p)}{(t-s)p_{0}(1+\hat{p}\cdot \omega)^{\frac{1}{2}}}dp d\sigma(\omega)\lesssim \int_{C_{t,x}} \frac{|B|\Phi_{-1}(s,x+(t-s)\omega)}{t-s} d\sigma(\omega)
\eea
But since $|B| \leq |K|$, we finally obtain:
\bea
(|E_{S,1}| + |B_{S,1}|)(t,x) \lesssim \int_{C_{t,x}} \frac{|K|\Phi_{-1}(s,x+(t-s)\omega)}{t-s} d\sigma(\omega)
\eea
which is (\ref{ks2ls}).
Recall that this is precisely the representation formula for the inhomogeneous wave equation of the form:
$$
\square u = |K|\Phi_{-1}; \ u|_{t=0} = \rd_{t}u|_{t=0}=0
$$
Finally, our last term has the following bound from Proposition 3.4 in \cite{L-S}:
\bea
(|E_{S,2}| + |B_{S,2}|)(t,x) \lesssim  \int_{C_{t,x}}\int_{\mathbb{R}^{3}} \frac{|K_{g}|f(s,x+(t-s)\omega)}{(t-s)p_{0}(1+\hat{p}\cdot \omega)} dp d\sigma(\omega)
\eea
where $|K_{g}|^{2} = |E\cdot \omega|^{2} + |B \cdot \omega|^{2} + |E - \omega \times B|^{2} + |B + \omega\times E|^{2}$. Recall the conservation law $\|K_{g}\|_{L^{2}(C_{t,x})} \lesssim 1$ from Proposition 2.2 in \cite{L-S} and use H\"older's inequality to obtain:

\bea
(|E_{S,2}| + |B_{S,2}|)(t,x) \lesssim  \bigg(\int_{C_{t,x}}\bigg(\int_{\mathbb{R}^{3}} \frac{f(s,x+(t-s)\omega)}{(t-s)p_{0}(1+\hat{p}\cdot \omega)} dp\bigg)^{2} d\sigma(\omega)\bigg)^{\frac{1}{2}}
\eea
Finally, using the same change of variables as in (\ref{changekt}), we get (\ref{ks2ls}).
\end{proof}

\section{New Bounds on $|K|$}

From this point in the paper, we adopt the convention that $\rho +$ denotes some appropriate number $\rho + \epsilon$ where $\epsilon > 0$ is very small, $\epsilon \ll 1$. Note that the size of $\epsilon$ may vary depending on the term, but the key point is that $\epsilon$ is appropriately small in each of the estimates below. Similarly, we let $\rho -$ denote some appropriate number $\rho - \epsilon$ for $\epsilon \ll 1$ chosen to be appropriately small.

\begin{prop}\label{newKTest}
Given $1\leq m \leq 3$, $\frac{3}{mq}-\frac{3}{q} > -1$ and $\frac{3m-1}{2m} \leq q \leq \infty$, we have the estimate:
\bea
\|K_{T}\|_{L^{\infty}_{t}L^{mq}_{x}} \lesssim \|\sigma_{-1}\|_{L^{\infty}_{t}L^{q}_{x}}
\eea
\end{prop}

\begin{proof}
By (\ref{ktls}), we can apply (\ref{OpEst}) for $\alpha =2$ to $|K_{T}|$.
\end{proof}

In particular if $m=2$, then we need $-\frac{3}{2q} > -1$ (or $q>\frac{3}{2}$) and $q \geq \frac{5}{4}$. Hence, we have for $q > 2$:
\bea\label{KTcor}
\|K_{T}\|_{L^{\infty}_{t}L^{4+}_{x}} \lesssim \|\sigma_{-1}\|_{L^{\infty}_{t}L^{2+}_{x}}
\eea

\begin{prop}\label{newKS2est}

Given $1\leq m \leq 3$, $\frac{3}{mq}-\frac{3}{q} > -1$ and $\frac{3m-1}{2m} \leq q \leq \infty$,
\bea
\|K_{S,2}\|_{L^{\infty}_{t}L^{2mq}_{x}} \lesssim \|\sigma_{-1}\|_{L^{\infty}_{t}L^{2q}_{x}}
\eea
\end{prop}

\begin{proof}
By (\ref{ks2ls}):
\bea
\|K_{S,2}\|_{L^{\infty}_{t}L^{2mq}_{x}} \lesssim \|W_{2}((\sigma_{-1})^{2})^{\frac{1}{2}}\|_{L^{\infty}_{t}L^{2mq}_{x}} =  \|W_{2}((\sigma_{-1})^{2})\|_{L^{\infty}_{t}L^{mq}_{x}}^{\frac{1}{2}}
\eea
We can apply (\ref{OpEst}) now to get:
\bea
\|K_{S,2}\|_{L^{\infty}_{t}L^{2mq}_{x}} \lesssim  \|(\sigma_{-1})^{2}\|_{L^{\infty}_{t}L^{q}_{x}}^{\frac{1}{2}} = \|\sigma_{-1}\|_{L^{\infty}_{t}L^{2q}_{x}}
\eea
\end{proof}

For reasons that will be clear in Section 9, we use Proposition \ref{newKTest} and Proposition \ref{newKS2est} to bound the quantities $\|K_{T}\|_{L^{\infty}_{t}L^{4+}_{x}}$ and $\|K_{S,2}\|_{L^{\infty}_{t}L^{4+}_{x}}$. Using Proposition \ref{newKS2est}, we can compute for $|K_{S,2}|$:
\bea\label{usableKS2}
\|K_{S,2}\|_{L^{\infty}_{t}L^{4+}_{x}}\lesssim \|\sigma_{-1}\|_{L^{\infty}_{t}L^{\frac{12}{5}+}_{x}}
\eea
where we used $m=\frac{5}{3}$ and $q=\frac{6}{5}+$. In particular, setting $q = \frac{6+\epsilon}{5}$ for $\epsilon \ll 1$, we see that $m$ and $q$ satisfy the conditions of Proposition \ref{newKS2est}. The explicit estimate written in (\ref{usableKS2}) is
$$ \|K_{S,2}\|_{L^{\infty}_{t}L^{4+\frac{2\epsilon}{3}}_{x}}\lesssim \|\sigma_{-1}\|_{L^{\infty}_{t}L^{\frac{12}{5}+\frac{2\epsilon}{5}}_{x}}.$$
Similarly, by Proposition \ref{newKTest}, we can compute for $|K_{T}|$:
\bea\label{usableKT}
\|K_{T}\|_{L^{\infty}_{t}L^{4+}_{x}}\lesssim \|\sigma_{-1}\|_{L^{\infty}_{t}L^{\frac{12}{5}+}_{x}}
\eea
where we used $ m = \frac{5}{3}$ and $q = \frac{12}{5}+$. Note that this is not the lowest Lebesgue norm exponent that can be chosen for $\sigma_{-1}$. However, we do not have a better bound in the $K_{S,2}$ estimate, so a better estimate on the $K_{T}$ term is not useful.
\\
\\
Finally, we employ an iteration argument using Strichartz estimates for the inhomogeneous wave equation to gain bounds on $K_{S,1}$. For the remainder of the paper, assume that
\bea\label{assumption}
\|\Phi_{-1}\|_{L^{\infty}_{t}L^{2}_{x}}\lesssim 1
\eea

\begin{prop}
 We have the following bound on $K_{S,1}$ assuming (\ref{assumption}):
\bea\label{newKS1est}
\|K_{S,1}\|_{L^{\infty}_{t}([0,T);L_{x}^{4+})}\lesssim 1 + \|\sigma_{-1}\|_{L_{t}^{\infty}([0,T);L^{\frac{12}{5}+}_{x})}
\eea
\end{prop}

\begin{proof}
For $\gamma \in (0,1)$, we obtain by (\ref{ks1ls}) for some interval $[a,b]\subset [0,T)$:
\bea
\|K_{S,1}\|_{L^{\frac{2}{\gamma}}_{t}L^{\frac{2}{1-\gamma}}_{x}([a,b]\times \mathbb{R}^{3})} \lesssim \|\square^{-1}(|K|\Phi_{-1})\|_{L^{\frac{2}{\gamma}}_{t}L^{\frac{2}{1-\gamma}}_{x}([a,b]\times \mathbb{R}^{3})}
\eea
(Note that we will set $\gamma = \frac{1}{2}+$ later in the proof.)
Applying the triangle inequality to the decomposition $|K| \leq |K_{0}| + |K_{T}| + |K_{S,1}| + |K_{S,2}|$ and extending the interval $[a,b]$ to $[0,T)$ on certain terms, we obtain:
\begin{multline}
\|K_{S,1}\|_{L^{\frac{2}{\gamma}}_{t}L^{\frac{2}{1-\gamma}}_{x}([a,b]\times \mathbb{R}^{3})} 
\\
\leq \|\square^{-1}(|K_{0}|\Phi_{-1})\|_{L^{\frac{2}{\gamma}}_{t}L^{\frac{2}{1-\gamma}}_{x}([0,T)\times \mathbb{R}^{3})} +\|\square^{-1}(|K_{T}|\Phi_{-1})\|_{L^{\frac{2}{\gamma}}_{t}L^{\frac{2}{1-\gamma}}_{x}([0,T)\times \mathbb{R}^{3})} 
\\
+\|\square^{-1}(|K_{S,2}|\Phi_{-1})\|_{L^{\frac{2}{\gamma}}_{t}L^{\frac{2}{1-\gamma}}_{x}([0,T)\times \mathbb{R}^{3})} +\|\square^{-1}(|K_{S,1}|\Phi_{-1})\|_{L^{\frac{2}{\gamma}}_{t}L^{\frac{2}{1-\gamma}}_{x}([a,b]\times \mathbb{R}^{3})}
\end{multline}

Note that we now replace the $\lesssim$ symbol with some explicit constant $\tilde{C}$. From here, since $\square^{-1}$ is the solution operator to the inhomogeneous wave equation on the interval $[0,T)$ with zero initial data as expressed in (\ref{wave}), we know from Theorem \ref{Strichartz Estimates} that
\bea\label{KTstrichartz}
\|\square^{-1}(|K_{T}|\Phi_{-1})\|_{L^{\frac{2}{\gamma}}_{t}L^{\frac{2}{1-\gamma}}_{x}([0,T)\times \mathbb{R}^{3})} \leq C_{\gamma} \||K_{T}|\Phi_{-1}\|_{L^{\frac{2}{1+\gamma}}_{t}L^{\frac{2}{2-\gamma}}_{x}([0,T)\times \mathbb{R}^{3})}
\eea
and similarly for $|K_{0}|$ and $|K_{S,2}|$. Next, since  $\|\Phi_{-1}\|_{L^{2}_{t}L^{2}_{x}([0,T)\times \mathbb{R}^{3})}\lesssim 1$ by (\ref{assumption}) and $\frac{2}{1-\gamma} = 4+$ by the assumption that $\gamma = \frac{1}{2}+$, we can apply H\"older's inequality and (\ref{usableKT}) to (\ref{KTstrichartz}) to get:
\begin{align*}
\||K_{T}|\Phi_{-1}\|_{L^{\frac{2}{1+\gamma}}_{t}L^{\frac{2}{2-\gamma}}_{x}([0,T)\times \mathbb{R}^{3})} & \leq \||K_{T}|\|_{L^{\frac{2}{\gamma}}_{t}L^{\frac{2}{1-\gamma}}_{x}([0,T)\times \mathbb{R}^{3})}\|\Phi_{-1}\|_{L^{2}_{t}L^{2}_{x}([0,T)\times \mathbb{R}^{3})}
\\
& \lesssim\||K_{T}|\|_{L^{\frac{2}{\gamma}}_{t}L^{4+}_{x}([0,T)\times \mathbb{R}^{3})} 
\\
&\lesssim \|\sigma_{-1}\|_{L^{\infty}_{t}L^{\frac{12}{5}+}_{x}}
\end{align*}

We obtain the same bound for the $|K_{S,2}|$ term. The $|K_{0}|$ term can be bounded by a constant since $|K_{0}|$ depends only on the initial data of the system. Summarizing, there exists a constant $C$:
\bea\label{useful}
\|K_{S,1}\|_{L^{\frac{2}{\gamma}}_{t}L^{\frac{2}{1-\gamma}}_{x}([a,b]\times \mathbb{R}^{3})} \leq CC_{\gamma}(1+\|\sigma_{-1}\|_{L^{\infty}_{t}L^{\frac{12}{5}+}_{x}}) +C\|\square^{-1}(|K_{S,1}|\Phi_{-1})\|_{L^{\frac{2}{\gamma}}_{t}L^{\frac{2}{1-\gamma}}_{x}([a,b]\times \mathbb{R}^{3})}
\eea

Now, let us set $u = \square^{-1}(|K_{S,1}|\Phi_{-1})$ for convenience of notation. Then, by Stichartz estimates and H\"older's inequality, we have the following fact:
\begin{multline}\label{Strichartzfact1}
\|u\|_{L^{\frac{2}{\gamma}}_{t}L^{\frac{2}{1-\gamma}}_{x}([a,b]\times \mathbb{R}^{3})} \leq C_{\gamma}\Big(\|u(a)\|_{\dot{H}^{\gamma}_{x}(\mathbb{R}^{3})} + \|\partial_{t}u(a)\|_{\dot{H}^{\gamma-1}_{x}(\mathbb{R}^{3})} 
\\ + \||K_{S,1}|\|_{L^{\frac{2}{\gamma}}_{t}L^{\frac{2}{1-\gamma}}_{x}([0,T)\times \mathbb{R}^{3})}\|\Phi_{-1}\|_{L^{2}_{t}L^{2}_{x}([0,T)\times \mathbb{R}^{3})}\Big).
\end{multline}

Next, due to (\ref{assumption}), we can choose a partition $0=T_{0}< T_{1}<T_{2}<\ldots<T_{N}=T$ of the interval $[0,T]$ such that:
\bea\label{partitionphi}
\|\Phi_{-1}\|_{L^{2}_{t}L^{2}_{x}([T_{j},T_{j+1}]\times \mathbb{R}^{3})} \leq \frac{1}{2CC_{\gamma}}
\eea
for $j=0,1,\ldots , N-1$.

Hence, by (\ref{Strichartzfact1}) and (\ref{useful}), we obtain:
\begin{multline}
\|K_{S,1}\|_{L^{\frac{2}{\gamma}}_{t}L^{\frac{2}{1-\gamma}}_{x}([T_{j},T_{j+1}]\times \mathbb{R}^{3})} \leq CC_{\gamma}(1+\|\sigma_{-1}\|_{L^{\infty}_{t}L^{\frac{12}{5}+}_{x}} + \|u(T_{j})\|_{\dot{H}^{\gamma}_{x}(\mathbb{R}^{3})} + \|\partial_{t}u(T_{j})\|_{\dot{H}^{\gamma-1}_{x}(\mathbb{R}^{3})})
\\
+\frac{1}{2}\||K_{S,1}|\|_{L^{\frac{2}{\gamma}}_{t}L^{\frac{2}{1-\gamma}}_{x}([T_{j},T_{j+1}]\times \mathbb{R}^{3})}.
\end{multline}
This implies that for any $j = 0,1,\ldots N-1$, we have the inequality:
\bea\label{iterative1}
\|K_{S,1}\|_{L^{\frac{2}{\gamma}}_{t}L^{\frac{2}{1-\gamma}}_{x}([T_{j},T_{j+1}]\times \mathbb{R}^{3})} \leq 2CC_{\gamma}(1+\|\sigma_{-1}\|_{L^{\infty}_{t}L^{\frac{12}{5}+}_{x}} +  \|u(T_{j})\|_{\dot{H}^{\gamma}_{x}(\mathbb{R}^{3})} + \|\partial_{t}u(T_{j})\|_{\dot{H}^{\gamma-1}_{x}(\mathbb{R}^{3})}).
\eea

%Using the $L^{\infty}$ norm in time, bound $\|u(T_{j})\|_{\dot{H}^{\gamma}_{x}(\mathbb{R}^{3})} + \|\partial_{t}u(T_{j})\|_{\dot{H}^{\gamma-1}_{x}(\mathbb{R}^{3})}$ terms in (\ref{iterative1}) by $$\|u\|_{C_{t}\dot{H}^{\gamma}_{x}([T_{j-1},T_{j}]\times\mathbb{R}^{3})} + \|\partial_{t}u\|_{C_{t}\dot{H}^{\gamma-1}_{x}([T_{j-1},T_{j}]\times\mathbb{R}^{3})}$$.
Using Strichartz estimates again, we obtain the bound:
\begin{multline}\label{Strichartzfact2}
\|u(T_{j})\|_{\dot{H}^{\gamma}_{x}(\mathbb{R}^{3})} + \|\partial_{t}u(T_{j})\|_{\dot{H}^{\gamma-1}_{x}(\mathbb{R}^{3})} \leq C_{\gamma} \Big(\|u(T_{j-1})\|_{\dot{H}^{\gamma}_{x}(\mathbb{R}^{3})} 
+ \|\partial_{t}u(T_{j-1})\|_{\dot{H}^{\gamma-1}_{x}(\mathbb{R}^{3})} 
\\+ \||K_{S,1}|\|_{L^{\frac{2}{\gamma}}_{t}L^{\frac{2}{1-\gamma}}_{x}([T_{j-1},T_{j})\times \mathbb{R}^{3})}\|\Phi_{-1}\|_{L^{2}_{t}L^{2}_{x}([T_{j-1},T_{j})\times \mathbb{R}^{3})}\Big).
\end{multline}

We now apply H\"older's inequality and the bound (\ref{partitionphi}) to (\ref{Strichartzfact2}) to get that:
\begin{multline}
\|u(T_{j})\|_{\dot{H}^{\gamma}_{x}(\mathbb{R}^{3})} + \|\partial_{t}u(T_{j})\|_{\dot{H}^{\gamma-1}_{x}(\mathbb{R}^{3})}  
\\
\leq C_{\gamma} \Big(\|u(T_{j-1})\|_{\dot{H}^{\gamma}_{x}(\mathbb{R}^{3})} 
+ \|\partial_{t}u(T_{j-1})\|_{\dot{H}^{\gamma-1}_{x}(\mathbb{R}^{3})} 
+ \frac{1}{2CC_{\gamma}}\||K_{S,1}|\|_{L^{\frac{2}{\gamma}}_{t}L^{\frac{2}{1-\gamma}}_{x}([T_{j-1},T_{j})\times \mathbb{R}^{3})}\Big).
\end{multline}

Next, by the estimate (\ref{iterative1}), we obtain that:
\begin{multline}
\|u(T_{j})\|_{\dot{H}^{\gamma}_{x}(\mathbb{R}^{3})} + \|\partial_{t}u(T_{j})\|_{\dot{H}^{\gamma-1}_{x}(\mathbb{R}^{3})}  
\\
\leq C_{\gamma} \Big(\|u(T_{j-1})\|_{\dot{H}^{\gamma}_{x}(\mathbb{R}^{3})} 
+ \|\partial_{t}u(T_{j-1})\|_{\dot{H}^{\gamma-1}_{x}(\mathbb{R}^{3})}\Big) 
+ \frac{1}{2C}\Big( 2CC_{\gamma} (1+\|\sigma_{-1}\|_{L^{\infty}_{t}L^{\frac{12}{5}+}_{x}}   
\\
+\|u(T_{j-1})\|_{\dot{H}^{\gamma}_{x}(\mathbb{R}^{3})} + \|\partial_{t}u(T_{j-1})\|_{\dot{H}^{\gamma-1}_{x}(\mathbb{R}^{3})}) \Big).
\end{multline}

Finally, it follows that:
\begin{multline}\label{iterative3}
\|u(T_{j})\|_{\dot{H}^{\gamma}_{x}(\mathbb{R}^{3})} + \|\partial_{t}u(T_{j})\|_{\dot{H}^{\gamma-1}_{x}(\mathbb{R}^{3})}  
\\
\leq 2C_{\gamma} \Big(\|u(T_{j-1})\|_{\dot{H}^{\gamma}_{x}(\mathbb{R}^{3})} 
+ \|\partial_{t}u(T_{j-1})\|_{\dot{H}^{\gamma-1}_{x}(\mathbb{R}^{3})}\Big) 
+  C_{\gamma} (1+\|\sigma_{-1}\|_{L^{\infty}_{t}L^{\frac{12}{5}+}_{x}}).
\end{multline}

Notice that $u(0) = \partial_{t}u(0)= 0$. Thus, performing an iteration of the above estimate (\ref{iterative3}), we obtain for any $k \in \{0,1,\ldots,N-1\}$:

$$\|u(T_{k})\|_{\dot{H}^{\gamma}_{x}(\mathbb{R}^{3})} + \|\partial_{t}u(T_{k})\|_{\dot{H}^{\gamma-1}_{x}(\mathbb{R}^{3})} \leq \sum_{j=0}^{k-1}(2C_{\gamma})^{k-1-j}(1+\|\sigma_{-1}\|_{L^{\infty}_{t}L^{\frac{12}{5}+}_{x}}).$$
Hence by (\ref{iterative1}), it follows that:
\bea\label{Kjest}
\|K_{S,1}\|_{L^{\frac{2}{\gamma}}_{t}L^{\frac{2}{1-\gamma}}_{x}([T_{k},T_{k+1}]\times \mathbb{R}^{3})} \leq 2CC_{\gamma}\Big(1+\|\sigma_{-1}\|_{L^{\infty}_{t}L^{\frac{12}{5}+}_{x}} +  \sum_{j=0}^{k-1}(2C_{\gamma})^{k-1-j}(1+\|\sigma_{-1}\|_{L^{\infty}_{t}L^{\frac{12}{5}+}_{x}})\Big).
\eea

Using the triangle inequality and summing (\ref{Kjest}) over $k = 0, 1, \ldots ,N-1$ and noting that $N$ is some finite positive integer depending on $\|\Phi_{-1}\|_{L^{\infty}_{t}L^{2}_{x}}$, we get
\begin{align*}
\|K_{S,1}\|_{L^{\frac{2}{\gamma}}_{t}L^{\frac{2}{1-\gamma}}_{x}([0,T)\times \mathbb{R}^{3})} & \leq \sum\limits_{k=0}^{N-1}2CC_{\gamma}\Big(1+\|\sigma_{-1}\|_{L^{\infty}_{t}L^{\frac{12}{5}+}_{x}} +  \sum_{j=0}^{k-1}(2C_{\gamma})^{k-1-j}(1+\|\sigma_{-1}\|_{L^{\infty}_{t}L^{\frac{12}{5}+}_{x}})\Big) \\
& \lesssim 1 + \|\sigma_{-1}\|_{L^{\infty}_{t}L^{\frac{12}{5}+}_{x}}.
\end{align*}
Since $\frac{2}{1-\gamma} = 4+$, we obtain the desired estimate (\ref{newKS1est}).
\end{proof}

\section{Pallard's Decomposition and Bounding P(T)}

In this section, we first recall the decomposition method in \cite{refinedPallard} and then apply the above estimates to gain a bound on the size of the momentum support of $f$, which we will denote by:
\bea\label{mometumsupportfunction}
P(T) \eqdef 1 + \sup\{p\in\mathbb{R}^{3}|\exists (t,x) \in [0,T)\times\mathbb{R}^{3} \text{such that} f(t,x,p)\neq 0\}
\eea
By the method of characteristics:
\bea
\frac{dV}{ds}(s;t,x,p')= E(s,X(s;t,x,p'))+\hat{V}(s;t,x,p')\times B(s,X(s;t,x,p'))
\eea

Taking the Eucliean inner product with $\hat{V}(s;t,x,p')$ on both sides and then integrating in time, we obtain:
\bea
\sqrt{1+|V(s;t,x,p')|^{2}} = \sqrt{1+|V(0;t,x,p')|^{2}} + \int_{0}^{T}E(s,X(s;t,x,p'))\cdot \hat{V}(s;t,x,p') ds
\eea

First, for $i=1,2,3$ and $K_{j} = E_{j} + (\hat{p} \times B)_{j}$, we can decompose the electric field:
\begin{multline}\label{EdecompPallard}
E_{i}(t,x)= E^{(0)}_{i}(x) + \int_{\mathbb{R}^{3}}{(\frac{(1-|\hat{p}|^{2})(x_{i}-t\hat{p})}{(t-\hat{p}\cdot x)^{2}})Y \star_{t,x} (f\chi_{t\geq 0}) dp }
\\
- \sum_{j=1}^{3}\int_{\mathbb{R}^{3}}{(\frac{[t(t-\hat{p}\cdot x)(\hat{p}_{i}\hat{p}_{j}-e_{i})+(x_{i}-t\hat{p}_{i})(x_{j}-(\hat{p}\cdot x)\hat{p}_{j})]}{p_{0}(t-\hat{p}\cdot x)^{2}})\star_{t,x} (K_{j}f\chi_{t\geq 0}) dp}
\\
\eqdef E^{(0)}_{i}(x) + F_{i}(t,x) + G_{i}(t,x)
\end{multline}
where $e_{i}$ is the unit vector with all entries equal to $0$ except for the i$^{th}$ entry which is equal to $1$. Also, the double convolution $\star_{t,x}$ is a binary operation defined by:
\bea
f_{1}\star_{t,x}f_{2} = \int_{\mathbb{R}\times\mathbb{R}^{3}}{f_{1}(t-s,x-y) f_{2}(s,y) \ ds \ dy}
\eea
and
\bea
Y \eqdef (4\pi t)^{-1}\delta_{|x|=t}
\eea

Following the scheme of \cite{refinedPallard}, we can decompose the characteristic integral of the electric field into:
\bea\label{Pallard1}
\int_{0}^{T}E(s,X(s;t,x,p'))\cdot \hat{V}(s;t,x,p') ds = I_{0} + I_{F} + I_{G}
\eea
where $I_{0}$ depends only on the initial data term $E^{(0)}$ and
\bea\label{IF}
I_{F} \eqdef \int_{0}^{T}F(s,X(s;t,x,p'))\cdot \hat{V}(s;t,x,p') ds
\eea
and 
\bea\label{IG}
I_{G} \eqdef \int_{0}^{T}G(s,X(s;t,x,p'))\cdot \hat{V}(s;t,x,p') ds
\eea

Pallard then bounds $I_{G}$ by:
\bea\label{IG1}
|I_{G}| \lesssim \int_{0}^{T} \int_{s}^{T} \int_{|y|=t-s}\int_{\mathbb{R}^{3}_{p}} \frac{(f|K|)(s,X(t)-y,p)}{p_{0}(1-\hat{p}\cdot \omega)}\Big(\sqrt{1-\hat{V}(t)\cdot\omega}\Big) dp\frac{d\sigma(y) dt}{4\pi|t-s|}ds
\eea

From here, Pallard \cite{refinedPallard} bounds the integral $$\int_{\mathbb{R}^{3}_{p}} \frac{(f|K|)(s,X(t)-y,p)}{p_{0}(1-\hat{p}\cdot \omega)} dp$$ using the term $ m(t,x) \eqdef \int_{\mathbb{R}^{3}}p_{0}f(t,x,p) dp$. Instead, we preserve the singularity in the denominator:
\bea\label{newpallard}
|I_{G}| \lesssim \int_{0}^{T} \int_{s}^{T} \int_{|y|=t-s} (\sigma_{-1}|K|)(s,X(t)-y)\Big(\sqrt{1-\hat{V}(t)\cdot\omega}\Big) \frac{d\sigma(y) dt}{4\pi|t-s|}ds
\eea

Split the integral into $I_{G} \lesssim I^{'}_{G} + I^{''}_{G}$ as follows:
\begin{multline}\label{newpallard2}
|I_{G}| \lesssim \int_{0}^{T} \int_{s}^{s+\epsilon(s)} \int_{|y|=t-s} (\sigma_{-1}|K|)(s,X(t)-y)\Big(\sqrt{1-\hat{V}(t)\cdot\omega}\Big) \frac{d\sigma(y) dt}{4\pi|t-s|}ds \ \\ + \int_{0}^{T} \int_{s+\epsilon(s)}^{T} \int_{|y|=t-s} (\sigma_{-1}|K|)(s,X(t)-y)\Big(\sqrt{1-\hat{V}(t)\cdot\omega}\Big) \frac{d\sigma(y) dt}{4\pi|t-s|}ds
\end{multline}

for 
\bea\label{epsilon}
\epsilon(s) = \frac{T-s}{1 + P(s)^{8}}
\eea
Note that the power of $P(s)$ in  (\ref{epsilon}) is useful for bounding $I^{'}_{G}$ as in \cite{refinedPallard}. First, let us bound $I^{''}_{G}$. By computing using H\"older's inequality as in \cite{refinedPallard}:
\begin{multline}\label{IG''}
|I^{''}_{G}| \lesssim \int_{0}^{T} \Big|\int_{s+\epsilon(s)}^{T} \int_{|y|=t-s} (\sigma_{-1}|K|)^{\frac{3}{2}}(s,X(t)-y)(1-\hat{V}(t)\cdot\omega) d\sigma(y) dt\Big|^{\frac{2}{3}} \\ \times \Big(\int_{s+\epsilon(s)}^{T}\int_{|y|=t-s}((1-\hat{V}(t)\cdot\omega)^{-\frac{1}{6}})^{3}d\sigma dt\Big)^{\frac{1}{3}} ds
\\
\lesssim \int_{0}^{T} \Big|\int_{s+\epsilon(s)}^{T} \int_{|y|=t-s} (\sigma_{-1}|K|)^{\frac{3}{2}}(s,X(t)-y)(1-\hat{V}(t)\cdot\omega) d\sigma(y) dt\Big|^{\frac{2}{3}} \ln^{\frac{1}{3}}\Big(\frac{T-s}{\epsilon(s)}\Big) ds
\end{multline}

Setting $\omega = \omega(\theta,\phi) = (\sin\theta \cos\phi,\sin\theta\sin\phi,\cos\theta)$:
\begin{multline}\label{114}
\int_{s+\epsilon(s)}^{T} \int_{|y|=t-s} (\sigma_{-1}|K|)^{\frac{3}{2}}(s,X(t)-y)(1-\hat{V}(t)\cdot\omega) d\sigma(y) dt
\\
 = \int_{s+\epsilon(s)}^{T} \int_{|y|=t-s} (\sigma_{-1}|K|)^{\frac{3}{2}}(s,X(t)-(t-s)\omega(\theta,\phi))(1-\hat{V}(t)\cdot\omega(\theta,\phi)) (t-s)^{2} \sin\theta d\theta d\phi dt
\end{multline}

Consider the change of variables $\Psi : (s_{1},s_{2})\times(0,\pi)\times(0,2\pi) \rightarrow \Psi((s_{1},s_{2})\times(0,\pi)\times(0,2\pi))$ mapping $$(t,\theta,\phi) \mapsto X(t) - (t-s)\omega(\theta,\phi) \eqdef z$$

The Jacobian of this map is $J = (\hat{V}(t)\cdot\omega - 1) (t-s)^{2}\sin\theta$. Applying this change of variables to (\ref{114}) and inserting our choice of $\epsilon(s)$, we obtain:

\bea\label{IG''part2}
|I^{''}_{G}| \lesssim \int_{0}^{T} \Big|\int_{\Psi((s_{1},s_{2})\times(0,\pi)\times(0,2\pi))} (\sigma_{-1}|K|)^{\frac{3}{2}}(s,z) dz dt\Big| \ln^{\frac{1}{3}}\Big(1 + P(s) \Big) ds
\eea

Following \cite{refinedPallard} precisely, we also know that $I^{'}_{G}\lesssim 1$. (This is done through first applying H\"older's inequality to isolate the first term and then using conservation law $\|K\|_{L^{\infty}_{t}L^{2}_{x}}\lesssim 1$. Finally, by the definition of $\epsilon(s)$, the leftover integral is bounded.) Thus, we arrive at the estimate:
\bea\label{newIGest}
|I_{G}| \lesssim 1 + \|\sigma_{-1}|K|\ln^{\frac{1}{3}}(1+P(t))\|_{L^{1}_{t}L^{\frac{3}{2}}_{x} ([0,T]\times \mathbb{R}^{3})}
\eea

Next, we recognize that $F$ is equivalent to our $E_{T}$ term as expressed in (\ref{ET}). Thus, using the proof of Proposition \ref{Jacobian}:

\bea\label{IFest}
|I_{F}| \lesssim \|\sigma_{-1}\|_{L^{\infty}_{t}L^{3+}_{x}}
\eea

In conclusion:

\begin{prop}\label{finalPTest}
By (\ref{newIGest}) and (\ref{IFest}), we have the following bound for $P(T)$:

\bea\label{finalPT}
|P(T)| \lesssim 1 + \|\sigma_{-1}\|_{L^{\infty}_{t}L^{3+}_{x}} + \|\sigma_{-1}|K|\ln^{\frac{1}{3}}(1+P(t))\|_{L^{1}_{t}L^{\frac{3}{2}}_{x} ([0,T]\times \mathbb{R}^{3})}
\eea
\end{prop}

\section{Moment Bounds}

We conclude by applying the estimates given by (\ref{usableKS2}), (\ref{usableKT}) and (\ref{newKS1est}) on $|K|$ under the assumption that $\|\Phi_{-1}\|_{L^{\infty}_{t}L^{2}_{x}} \lesssim 1$:

\begin{align*}
\|\sigma_{-1}|K|\ln^{\frac{1}{3}}(1+P(t))\|_{L^{1}_{t}L^{\frac{3}{2}}_{x} ([0,T]\times \mathbb{R}^{3})} &\leq \ln^{\frac{1}{3}}(1+P(T))
\||K|\|_{L^{1}_{t}L^{4+}_{x}}\|\sigma_{-1}\|_{L^{\infty}_{t}L^{\frac{12}{5}-}_{x}} 
\\
&\lesssim \ln^{\frac{1}{3}}(1+P(T))(1 + \|\sigma_{-1}\|_{L^{\infty}_{t}L^{\frac{12}{5}+}_{x}})\|\sigma_{-1}\|_{L^{\infty}_{t}L^{\frac{12}{5}-}_{x}}
\end{align*}
Notice that our choice of H\"older exponents used in the first line above allow for the Lebesgue norm exponents on both terms involving $\sigma_{-1}$ to be approximately equivalent to $\frac{12}{5}$. This choice of H\"older exponents simplifies our computation. Other choices yield similar results. We can now use Lemma \ref{sigmainterpolation} to bound each term in (\ref{finalPT}) for some $\beta >0$ arbitrarily small:

\bea\label{sigma1}
\|\sigma_{-1}\|_{L^{\infty}_{t}L^{\frac{12}{5}-}_{x}} \lesssim \|p_{0}^{\frac{24}{5r} - 1}f\|_{L^{\infty}_{t}L^{r}_{x}L^{1}_{p}}^{\frac{5r}{12}+}
\eea

\bea\label{sigma2}
\|\sigma_{-1}\|_{L^{\infty}_{t}L^{\frac{12}{5}+}_{x}} \lesssim \|p_{0}^{\frac{24}{5r} - 1+\beta}f\|_{L^{\infty}_{t}L^{r}_{x}L^{1}_{p}}^{\frac{5r}{12}-}
\eea

\bea\label{sigma3}
\|\sigma_{-1}\|_{L^{\infty}_{t}L^{3+}_{x}} \lesssim \|p_{0}^{\frac{6}{r}-1+\beta}f\|_{L^{\infty}_{t}L^{r}_{x}L^{1}_{p}}^{\frac{r}{3}-}
\eea

We can extract $\frac{12}{10r}-\delta$ power of $p_{0}$ for some $\delta >0$ arbitrarily small from each of (\ref{sigma1}) and (\ref{sigma2}) and $\frac{3}{r}-\delta$ power of $p_{0}$ from (\ref{sigma3}). Thus:

\bea\label{sigma1.1}
\|\sigma_{-1}\|_{L^{\infty}_{t}L^{\frac{12}{5}-}_{x}} \lesssim \|p_{0}^{\frac{18}{5r} - 1}f\|_{L^{\infty}_{t}L^{r}_{x}L^{1}_{p}}^{\frac{5r}{12}+}P(T)^{\frac{1}{2}-}
\eea

\bea\label{sigma2.1}
\|\sigma_{-1}\|_{L^{\infty}_{t}L^{\frac{12}{5}+}_{x}} \lesssim \|p_{0}^{\frac{18}{5r} - 1+\beta}f\|_{L^{\infty}_{t}L^{r}_{x}L^{1}_{p}}^{\frac{5r}{12}}P(T)^{\frac{1}{2}-}
\eea

\bea\label{sigma3.1}
\|\sigma_{-1}\|_{L^{\infty}_{t}L^{3+}_{x}} \lesssim \|p_{0}^{\frac{3}{r}-1+\beta}f\|_{L^{\infty}_{t}L^{r}_{x}L^{1}_{p}}^{\frac{r}{3}-}P(T)^{1-}
\eea
where $P(T)^{1-}$ indicates a power of $P(T)$ smaller than $1$ by an arbitrarily small amount. Assume that $\|p_{0}^{\frac{18}{5r} - 1+\beta}f\|_{L^{\infty}_{t}L^{r}_{x}L^{1}_{p}} \lesssim 1$. (Hence $\|p_{0}^{\frac{3}{r} - 1+\beta}f\|_{L^{\infty}_{t}L^{r}_{x}L^{1}_{p}} \lesssim \|p_{0}^{\frac{18}{5r} - 1+\beta}f\|_{L^{\infty}_{t}L^{r}_{x}L^{1}_{p}} \lesssim 1$.) 
\\
\\
Plugging these into (\ref{finalPT}), we obtain the bound:

\bea
P(T) \lesssim 1 + \ln^{\frac{1}{3}}(1+P(T))P(T)^{1-}
\eea
which implies that $P(T)\lesssim 1$ since $P(T) > 1$. Finally the last term we need to take care of is the assumption that $\|\Phi_{-1}\|_{L^{\infty}_{t}L^{2}_{x}} \lesssim 1 $. By employing similar proof to Lemma \ref{sigmainterpolation}, we see that:

\begin{prop}
Given $r\in [1,2]$, we have the estimate:

\bea\label{phiest}
\|\Phi_{-1}\|_{L^{2}_{x}} \lesssim \|p_{0}^{\alpha}f\|_{L^{r}_{x}L^{1}_{p}}^{\frac{r}{2}}
\eea

where $\alpha > \frac{2}{r}-1$.
\end{prop}

\begin{proof}
Fix some $\omega \in \mathbb{S}^{2}$ and let $r=\frac{2}{q}$. Then $\frac{q'}{2} \geq 1$  (since $\frac{1}{q}+\frac{1}{q'}=1$ and $q\geq 2$) and $\frac{1}{1 + \hat{p}\cdot\omega} \lesssim p_{0}^{2}$ implies:
\begin{align}
\int_{\mathbb{R}^{3}}{\frac{f(t,x,p)}{p_{0}(1+\hat{p}\cdot\omega)^{\frac{1}{2}}} dp} & \lesssim \Big(\int_{\mathbb{R}^{3}}{\frac{1}{p_{0}^{(\beta+1)q'}(1+\hat{p}\cdot\omega)^{\frac{q'}{2}}} dp}\Big)^{\frac{1}{q'}}\Big(\int_{\mathbb{R}^{3}}{p_{0}^{\beta q}f(t,x,p) dp}\Big)^{\frac{1}{q}} \|f\|_{L^{\infty}_{t,x,p}}^{\frac{q-1}{q}} \\
& \lesssim \Big(\int_{\mathbb{R}^{3}}{\frac{1}{p_{0}^{\beta q' + 2}(1+\hat{p}\cdot\omega)} dp}\Big)^{\frac{1}{q'}}\Big(\int_{\mathbb{R}^{3}}{p_{0}^{\beta q}f(t,x,p) dp}\Big)^{\frac{1}{q}}
\end{align}

By (\ref{integralbound}) in the proof of Lemma \ref{sigmainterpolation}, we know that the first integral on the right hand side is bounded by a constant when $\beta q' > 1$, i.e. $\beta q > q -1$. Taking the $L^{2}$ norm of this inequality:
\bea
\Bigg\|\int_{\mathbb{R}^{3}}{\frac{f(t,x,p)}{p_{0}(1+\hat{p}\cdot\omega)^{\frac{1}{2}}} dp}\Bigg\|_{L^{2}_{x}} \lesssim \Bigg\|\Big(\int_{\mathbb{R}^{3}}{p_{0}^{\beta q}f(t,x,p) dp}\Big)^{\frac{1}{q}}\Bigg\|_{L^{2}_{x}} = \Bigg\|\Big(\int_{\mathbb{R}^{3}}{p_{0}^{\beta q}f(t,x,p) dp}\Big)\Bigg\|_{L^{\frac{2}{q}}_{x}L^{1}_{p}}^{\frac{1}{q}}
\eea

Finally, setting $\alpha = \beta q$, we obtain that
$$ \alpha > q-1 = \frac{2}{r} -1.$$ Taking the maximum over all $\omega \in \mathbb{S}^{2}$ retains the same upper bound. Hence, this completes the proof.
\end{proof}

In particular, notice that for $1\leq r \leq 2$:
\bea
\|\Phi_{-1}\|_{L^{2}_{x}} \lesssim \|p_{0}^{\frac{2}{r}-1 + \beta}f\|_{L^{r}_{x}L^{1}_{p}} \lesssim \|p_{0}^{\frac{18}{5r} - 1+\beta}f\|_{L^{\infty}_{t}L^{r}_{x}L^{1}_{p}} \lesssim 1
\eea

Thus, if $\|p_{0}^{\frac{18}{5r} - 1+\beta}f\|_{L^{\infty}_{t}L^{r}_{x}L^{1}_{p}} \lesssim 1$, then $\|\Phi_{-1}\|_{L^{2}_{x}} \lesssim 1$. Hence, all of the terms in (\ref{finalPT}), which implicitly included the assumption $\|\Phi_{-1}\|_{L^{2}_{x}} \lesssim 1$, are bounded. Thus, indeed we do know that $P(T)\lesssim 1$. Thus, we can extend our local solution on the time interval $[0,T)$ to a larger time interval $[0,T+\epsilon]$. This concludes the proof of 
\bea
\|p_{0}^{\frac{18}{5r} - 1+\beta}f\|_{L^{\infty}_{t}L^{r}_{x}L^{1}_{p}} \lesssim 1
\eea
as a continuation criteria for $1\leq r \leq 2$.

\section{Proof of Theorem \ref{main2}}

In this section, we prove our final result, Theorem \ref{main2}. First, we state the following bounds analogous to \cite{L-S}. The inequality (\ref{singularitybound}) is proven analogously to Proposition 4.3 in \cite{L-S}, where we replace the fixed unit vector $e_{3}$ with the time-dependent unit vector $n_{3}(t)$. This change does not affect the proof because our inequality is pointwise in time. Before stating our main propositions, we define the following notation for vectors $v ,w \in \mathbb{R}^{3}$:

 $$\angle(v,\pm w) \eqdef \min\{\angle(v,w),\angle(v,-w)\},$$ which will be used throughout this section.

\begin{prop} For any $p\in\mathbb{R}^{3}$ and $\omega\in \mathbb{S}^{2}$:
	\bea\label{singularity}
	(1+\hat{p}\cdot\omega)^{-1} \lesssim \text{min}\{p_{0}^{2},(\angle(\frac{p}{|p|},-\omega))^{-2}\}
	\eea
	Further, if $\gamma = \tan^{-1}\big(\frac{p\cdot n_{2}(t)}{p\cdot n_{1}(t)}\big)$ and $p\in \text{supp}\{f\}$, then
	\bea\label{pnorm}
	|p| \lesssim \frac{\kappa(t,\gamma(p))}{\angle(\frac{p}{|p|},\pm n_{3}(t))}
	\eea
	Combining (\ref{singularity}) and (\ref{pnorm}), we obtain the following estimate
	for $p\in \text{supp}\{f\}$:
	\bea\label{singularitybound}
	(1+\hat{p}\cdot\omega)^{-1} \lesssim \text{min}\{\bigg(\frac{\kappa(t,\gamma(p))}{\angle(\frac{p}{|p|},\pm n_{3}(t))}\bigg)^{2},(\angle(\frac{p}{|p|},-\omega))^{-2}\}
	\eea
\end{prop}

Define $\omega^{(i)} = (\sin(\theta^{(i)})\cos(\phi^{(i)}),\sin(\theta^{(i)})\sin(\phi^{(i)}),\cos(\theta^{(i)}))$ where $\omega^{(i)}$ is the transformation of $\omega$ under a rotation matrix that takes $e_{i}$ to $n_{i}(T_{i})$. Thus, $\theta^{(i)} = \angle(n_{3}(T_{i}), \omega^{(i)})$. By similar arguments to Proposition 4.4 in \cite{L-S}, we obtain:

\begin{prop}\label{maintool}
We have the uniform estimate
	\begin{equation}
	\int_{\mathbb{R}^{3}}{\frac{f(s,x+r\omega^{(i)},p)}{p_{0}(1+\hat{p}\cdot\omega^{(i)})}dp} \lesssim \min \{P(s)^{2}\log(P(s)), \frac{A(s)^{4}\log(P(s))}{(\angle(n_{3}(s),\pm\omega^{(i)}))^{2}}\}
	\end{equation}
	for $s\in [T_{i},T_{i+1})$.
\end{prop}

\begin{proof}
We follow the proof of Proposition 4.4 in \cite{L-S}, emphasizing the steps in which we deviate from their proof. First, pick spherical coordinates $\theta_{(i)}, \phi_{(i)}$ such that $-\omega^{(i)}$ lies on the half-axis $\theta_{(i)} = 0$. Then, by (\ref{singularity}), we obtain the estimate
\bea\label{singularomegai}
(1+p\cdot\omega^{(i)})^{-1} \lesssim \text{min}\{p_{0}^{2}, (\theta_{(i)})^{-2}\}.
\eea
By the definition of $P(s)$, the particle density $f(s,x+r\omega^{(i)},p) = 0$ for $|p| > P(s)$. Thus, the conservation law $\|f\|_{L^{\infty}_{x,p}}\lesssim 1$ and the inequality (\ref{singularomegai}) imply that
\begin{align*}
\int_{\mathbb{R}^{3}}{\frac{f(s,x+r\omega^{(i)},p)}{p_{0}(1+\hat{p}\cdot\omega^{(i)})}dp} &\lesssim \int_{|p|\leq P(s)}{\frac{1}{p_{0}(1+\hat{p}\cdot\omega^{(i)})}dp} \\
&\lesssim \int_{0}^{P(s)}\int_{0}^{\pi}\int_{0}^{2\pi}{\frac{1}{p_{0}(1+\hat{p}\cdot\omega^{(i)})}d|p| \ d\theta_{(i)} \ d\phi_{(i)}} \\
&\lesssim \int_{0}^{P(s)}\int_{0}^{P(s)^{-1}}{p_{0}^{2} d|p| \ d\theta_{(i)}} + \int_{0}^{P(s)}\int_{P(s)^{-1}}^{\pi}{(\theta_{(i)})^{-2} d|p| \ d\theta_{(i)}} \\
&\lesssim P(s)^{2}\log(P(s)),
\end{align*}
which proves the first part of our proposition. We now move on to prove the second bound we need. Let $\beta_{i} = \angle(n_{3}(s),\pm\omega^{(i)})$. We partition the range $[-\frac{\pi}{2},\frac{\pi}{2}]$ of $\beta_{i}$ as in \cite{L-S}:
$$ I_{i} = \{\angle(n_{3}(s),\omega^{(i)})\leq \frac{\beta_{i}}{2} \} \cup I_{i} = \{\angle(n_{3}(s),-\omega^{(i)})\leq \frac{\beta_{i}}{2} \}$$
$$ II_{i} = \{\angle(n_{3}(s),\pm\omega^{(i)})\geq \frac{\beta_{i}}{2} \} \cap I_{i} = \{\angle(n_{3}(s),\pm\omega^{(i)})\geq \frac{\beta_{i}}{2} \}.$$
By the definition of $\beta_{i}$ and the triangle inequality:
$$ \angle (\frac{p}{|p|}, \omega^{(i)}) \geq |\angle (n_{3}(s), \omega^{(i)}) - \angle (\frac{p}{|p|}, n_{3}(s))|$$ if $\angle (n_{3}(s), \omega^{(i)}) \leq \frac{\beta}{2}$. Similarly, if $\angle (n_{3}(s), -\omega^{(i)}) \leq \frac{\beta}{2}$, then
$$ \angle (\frac{p}{|p|}, \omega^{(i)}) \geq |\angle (n_{3}(s), \omega^{(i)}) - \angle (\frac{p}{|p|}, n_{3}(s))|.$$ We can do the same estimate for $\angle (\frac{p}{|p|}, -\omega^{(i)})$, and hence,
$$ \angle (n_{3}(s), \pm\omega^{(i)}) \leq \frac{\beta}{2}.$$ By (\ref{singularitybound}), we now know that
$$(1+p\cdot\omega^{(i)})^{-1} \lesssim \beta^{-2}.$$ Using this estimate for region $I_{i}$ and defining the domains $D_{i}$ and $\tilde{D}_{i}$ as
$$ D_{i} = \{ p\in\mathbb{R}^{3} \ | \ \exists \ x\in\mathbb{R}^{3} \text{ such that } f(s,x,p)\neq 0\} $$ and
$$ D_{i} = \{ (p_{1},p_{2})\in\mathbb{R}^{2} \ | \ \exists \ x\in\mathbb{R}^{3}, p_{3}\in\mathbb{R} \text{ such that } f(s,x,p_{1},p_{2},p_{3})\neq 0\},$$ we obtain the following estimate on region $I_{i}$:
\begin{align*}
\int_{\mathbb{R}^{3}}{\frac{f(s,x+r\omega^{(i)},p)}{(1+p\cdot\omega^{(i)})^{-1}} dp} &\lesssim \beta^{-2} \int_{D_{i}}{\frac{1}{p_{0}} dp} \\
&\lesssim \beta^{-2} \int_{\tilde{D}_{i}}\int_{-P(s)}^{P(s)}{\frac{1}{\sqrt{1+p_{3}}} dp_{3} \ dp_{1} \ dp_{2}} \\
&\lesssim \beta^{-2}\log(P(s)) \int_{\tilde{D}_{i}}{dp_{1} \ dp_{2}} \\
&\lesssim \beta^{-2}\log(P(s)) \int_{0}^{2\pi}\int_{0}^{\kappa(s,\gamma)} u du d\gamma \\
&\lesssim  \beta^{-2}\log(P(s)) \|\kappa(s,\gamma)\|^{2}_{L^{4}_{\gamma}}
\end{align*}
(In the above, we used polar coordinate to compute the integral over $\tilde{D}$ and H\"older's inequality in $\gamma$ in the last step.) Thus, we have obtained the bound in region $I_{i}$:
$$ \int_{I_{i}}{\frac{f(s,x+r\omega^{(i)},p)}{p_{0}(1+p\cdot\omega^{(i)})} dp} \lesssim \beta^{-2}\log(P(s))A(s)^{2} \lesssim \frac{\log(P(s))A(s)^{4}}{(\angle(n_{3}(s),\pm\omega^{(i)}))^{2}}.$$ For region $II_{i}$, pick a system of polar coordinates $(\theta_{s},\phi_{s})$ such that $p\cdot n_{3}(s) = |p|\cos(\theta_{s})$, i.e. $\theta_{s} = \angle(p, n_{3}(s))$. Hence, by definition of $\beta$, we have that $\frac{\beta}{2} \leq \theta_{s} \leq \frac{\pi}{2}- \frac{\beta}{2}$ and by definition of $\gamma = \gamma(p)$, we also have that $\phi_{s} = \gamma(p)$. By (\ref{pnorm}), we have that
$$ |p| \lesssim \kappa(t,\phi_{s})(\theta_{s}^{-1} + (\pi-\theta_{s})^{-1}).$$ Using (\ref{singularitybound}), we obtain
\begin{multline*}
\int_{II_{i}}{\frac{f(s,x+r\omega^{(i)},p)}{p_{0}(1+p\cdot\omega^{(i)})} dp} \\ \lesssim \int_{0}^{2\pi} d\phi_{s} \int_{\frac{\beta}{2}}^{\frac{\pi}{2}-\frac{\beta}{2}} \sin(\theta_{s}) \ d\theta_{s}\int_{0}^{C\kappa(t,\phi_{s})(\theta_{s}^{-1} + (\pi-\theta_{s})^{-1})} |p| \kappa(t,\phi_{s})^{2}(\theta_{s}^{-2} + (\pi-\theta_{s})^{-2}) d|p| \\
\lesssim \beta^{-2}A(t)^{4} \lesssim \frac{\log(P(s))A(s)^{4}}{(\angle(n_{3}(s),\pm\omega^{(i)}))^{2}}
\end{multline*}
Summing the integrals over the domains $I_{i}$ and $I_{ii}$, we obtain the second bound we wanted. This completes our proof.
\end{proof}

In the above, $\angle(n_{3}(s),\pm\omega^{(i)}) \eqdef \min\{\angle(n_{3}(s),\omega^{(i)}),\angle(n_{3}(s),-\omega^{(i)})\}$. Notice that the above inequality is pointwise in time. Thus, the proof Proposition \ref{maintool} differs from the proof of Proposition 4.4 in \cite{L-S} only in that we replace the unit vector $e_{3} = (0,0,1)$ with $n_{3}(s)$ and $\omega$ with $\omega^{(i)}$. We now give modified arguments for momentum support on planes changing uniformly continuously in time.

\begin{prop}\label{ETestimate}
	
For $t\in [0,T)$:
\bea\label{etestimate}
|E_{T}(t,x)| + |B_{T}(t,x)|  \lesssim \log(P(t)) + (\log(P(t)))^{2}\int_{0}^{t}{A(s)^{4}ds}
\eea
\end{prop}

\begin{proof}
	Using the bound (\ref{ET}) and partitioning the time interval $$[0,t] = \cup_{0}^{n_{t}} ([T_{i},T_{i+1}]\cap [0,t])$$ as given by the conditions (\ref{condition1}) and (\ref{condition2}):
	\begin{equation*}
	\begin{split}
	|E_{T}(t,x)| + |B_{T}(t,x)|  &\lesssim \int_{C_{t,x}}{\int_{\mathbb{R}^{3}}{\frac{f(s,x+(t-s)\omega,p)}{(t-s)^{2}p_{0}^{2}(1+\hat{p}\cdot\omega)^{\frac{3}{2}}}dp \ d\omega}}  \\
	& = \int_{0}^{t}{\int_{0}^{2\pi}{\int_{0}^{\pi}{\int_{\mathbb{R}^{3}}{\frac{f(s,x+(t-s)\omega,p)}{p_{0}^{2}(1+\hat{p}\cdot\omega)^{\frac{3}{2}}}dp} \sin(\theta) \ d\theta} \ d\phi} \ ds}  \\
	& = \sum_{0}^{n_{t}}{\int_{T_{i}}^{T_{i+1}}{\int_{0}^{2\pi}{\int_{0}^{\pi}{\int_{\mathbb{R}^{3}}{\frac{f(s,x+(t-s)\omega^{(i)},p)}{p_{0}^{2}(1+\hat{p}\cdot\omega^{(i)})^{\frac{3}{2}}}dp} \  \sin(\theta^{(i)}) d\theta^{(i)}} \ d\phi^{(i)}} \ ds}}
	\end{split}
	\end{equation*}
We can divide the integral over $d\theta^{(i)}$ into three regions:
$$\int_{0}^{\pi}{\int_{\mathbb{R}^{3}}{\frac{f(s,x+(t-s)\omega^{(i)},p)}{p_{0}^{2}(1+\hat{p}\cdot\omega^{(i)})^{\frac{3}{2}}}dp} \ d\theta^{(i)}} \\ = A_{i}+B_{i}+C_{i}$$
	
	where $A_{i}$ is the integral over $[0,P(t)^{-1}]$, $B_{i}$ is the integral over $[P(t)^{-1}, \pi-P(t)^{-1}]$, and $C_{i}$ is the integral over $[\pi-P(t)^{-1}, \pi]$. We estimate each of these integrals using Proposition \ref{maintool}.
	
	\begin{equation*}
	\begin{split}
	B_{i} & = \int_{P(t)^{-1}}^{\pi-P(t)^{-1}}{\int_{\mathbb{R}^{3}}{\frac{f(s,x+(t-s)\omega^{(i)},p)}{p_{0}^{2}(1+\hat{p}\cdot\omega^{(i)})^{\frac{3}{2}}}dp} \ \sin(\theta^{(i)})d\theta^{(i)}} \\
	& \lesssim \int_{P(t)^{-1}}^{\pi-P(t)^{-1}}{\frac{A(s)^{4}\log(P(s))}{(\angle(n_{3}(s),\pm\omega^{(i)}))^{2}} \ \sin(\theta^{(i)})d\theta^{(i)}} \\
	& \lesssim \int_{P(t)^{-1}}^{\pi-P(t)^{-1}}{\frac{A(s)^{4}\log(P(s))}{(\theta^{(i)})^{2}} \ \sin(\theta^{(i)})d\theta^{(i)}} + \int_{P(t)^{-1}}^{\pi-P(t)^{-1}}{\frac{A(s)^{4}\log(P(s))}{(\pi -\theta^{(i)})^{2}} \ \sin(\pi - \theta^{(i)})d\theta^{(i)}}
	\end{split}
	\end{equation*}
	
	where in the third line we used the fact that $\sin(\theta^{(i)}) = \sin(\pi-\theta^{(i)})$ and we also used the following triangle inequality argument for angles:
	
	$$\angle(n_{3}(s),\pm\omega^{(i)}) \geq |\angle(n_{3}(T_{i}),\pm\omega^{(i)}) - \angle(n_{3}(s),n_{3}(T_{i}))|$$
	In the time interval $[T_{i}, T_{i+1}]$, we have that $\angle(n_{3}(s),n_{3}(T_{i})) < \frac{P(t)^{-1}}{4}$. Further, we are integrating over the interval $\theta^{(i)} = \angle(n_{3}(T_{i}),\omega^{(i)})\in [P(t)^{-1},\pi - P(t)^{-1}]$ and  $\pi - \theta^{(i)} = \angle(n_{3}(T_{i}),-\omega^{(i)})\in [P(t)^{-1},\pi - P(t)^{-1}]$. Thus,
	$$|\angle(n_{3}(T_{i}),\omega^{(i)}) - \angle(n_{3}(s),n_{3}(T_{i}))| \approx \theta^{(i)}$$
	and
	$$|\angle(n_{3}(T_{i}),-\omega^{(i)}) - \angle(n_{3}(s),n_{3}(T_{i}))| \approx \pi-\theta^{(i)}$$
	
	Evaluating the integral, we obtain:
	$$B_{i}\lesssim A(s)^{4}\log(P(t))^{2}$$
	and
	$$\sum_{0}^{n_{t}}{\int_{T_{i}}^{T_{i+1}}{B_{i} \ ds}} \lesssim \log(P(t))^{2}\int_{0}^{t}{A(s)^{4} ds}$$
	
	Next, evaluating $A_{i}$ and $C_{i}$ using the estimate $$\int{\frac{f(s,x+r\omega,p)}{p_{0}(1+\hat{p}\cdot\omega)}dp} \lesssim P(s)^{2}\log(P(s))$$ we obtain that
	\begin{align*}\label{aici}
	A_{i} & \lesssim \int_{0}^{P(t)^{-1}}{P(s)^{2}\log(P(s))} \sin(\theta^{(i)})d\theta^{(i)} \\
	& \lesssim \log(P(t))
	\end{align*}
	
	and similarly for $C_{i} \lesssim \log(P(t))$. Summing over $i = 1,\ldots, n_{t}$, we obtain our result.
\end{proof}

Next, we bound the $E_{S,1}+B_{S,1}$ term. To do so, we apply the argument directly from \cite{L-S}:

\begin{prop}
	For $s\in[0,T)$:
	\bea\label{newfbound}
	||\int_{\mathbb{R}^{3}}{f(s,x,p) \ dp}||_{L^{\infty}_{x}} \lesssim A(s)^{2}P(s)
	\eea
	
\end{prop}

\begin{proof}
	Consider coordinates on $\mathbb{R}^{3}$ such that $Q(s)$ is lies in the $(p_{1}, p_{2}, 0)$ plane. By the support of $f$ and since $f$ is a bounded function,
	
	$$||\int_{\mathbb{R}^{3}}{f \ dp}||_{L^{\infty}} \lesssim \int_{-P(s)}^{P(s)}{dp_{3}}\int_{0}^{2\pi}{d\gamma}\int_{0}^{\kappa(s,\gamma)}{r dr} \lesssim A(s)^{2}P(s)$$
	
\end{proof}

Since we still have the same bound (\ref{newfbound}) as in \cite{L-S}, the proof of Proposition 5.3 in \cite{L-S} follows exactly:

\begin{prop}
	
	For $t\in[0,T)$:
	
	\bea\label{ES1estimate}
	\int_{0}^{t}{|E_{S,1}| + |B_{S,1}| ds} \lesssim \sqrt{\log P(t)}\int_{0}^{t}{A(s)^{2}P(s)ds}
	\eea
\end{prop}

Finally, we have:

\begin{prop}
	For $t\in[0,T)$:
	\bea\label{es2inequality}
	|E_{S,2}| + |B_{S,2}| \lesssim P(t)\log P(t) + P(t)\log P(t) \Big(\int_{0}^{t}{A(s)^{8}ds}\Big)^{\frac{1}{2}}
	\eea
\end{prop}

\begin{proof}
	
	Applying H\"older's inequality to (\ref{ES1}):
	
	\bea\label{es2begin}	
	|E_{S,2}| + |B_{S,2}| \lesssim ||K_{g}||_{L^{2}(C_{t,x})} \Big(\int_{0}^{t}\int_{0}^{2\pi}\int_{0}^{\pi}\Big({\int_{\mathbb{R}^{3}}{\frac{f(s,x+(t-s)\omega,p)}{p_{0}(1+\hat{p}\cdot \omega)}dp}\Big)^{2}}\sin\theta d\theta d\phi ds\Big)^{\frac{1}{2}}
	\eea
	
	The $ ||K_{g}||_{L^{2}(C_{t,x})}$ term is uniformly bounded so we just have to get an estimate on the second term on the right. We apply the same decomposition as in the proof of Proposition \ref{ETestimate}. First, we split the integral over $\theta$ into three intervals and apply (\ref{maintool}) to the momentum integral to obtain the inequality:
	
	\bea\label{es2split}
	\int_{0}^{t}\int_{0}^{2\pi}\int_{0}^{\pi}\Big({\int_{\mathbb{R}^{3}}{\frac{f(s,x+(t-s)\omega,p)}{p_{0}(1+\hat{p}\cdot \omega)}dp}\Big)^{2}}\sin\theta d\theta d\phi ds \lesssim \sum_{i=0}^{n_{t}}{A_{i} + B_{i} + C_{i}}
	\eea
	
	where
	
	$$A_{i} = \int_{T_{i}}^{T_{i+1}}\int_{0}^{2\pi}\int_{0}^{P(t)^{-1}} P(s)^{4} \log(P(s))^{2}\sin\theta d\theta d\phi ds$$
	
	$$B_{i} = \int_{T_{i}}^{T_{i+1}}\int_{0}^{2\pi}\int_{P(t)^{-1}}^{\pi-P(t)^{-1}} \frac{A(s)^{8}\log(P(s))^{2}}{(\angle(n_{3}(s),\pm \omega^{(i)})^{4})} \sin\theta d\theta d\phi ds$$
	
	$$C_{i} = \int_{T_{i}}^{T_{i+1}}\int_{0}^{2\pi}\int_{\pi-P(t)^{-1}}^{\pi} P(s)^{4} \log(P(s))^{2}\sin(\pi-\theta) d\theta d\phi ds$$
	
	Now, we apply the same methods to bound $A_{i}$, $B_{i}$ and $C_{i}$ as in Proposition \ref{ETestimate} to obtain that:
	
	\bea\label{aibici}
	\sum_{i=0}^{n_{t}}{A_{i}+B_{i}+C_{i}} \lesssim P(t)^{2}\log(P(t))^{2} + P(t)^{2}\log(P(t))^{2}\int_{0}^{t}{A(s)^{8} ds}
	\eea
	
	Plugging (\ref{aibici}) into (\ref{es2begin}), we obtain our result.
\end{proof}

Notice that we have proven the same bounds on the fields $E$ and $B$ as found in \cite{L-S}. Thus, we can borrow the same proof from Proposition 6.1 in \cite{L-S} to obtain that $P(T) \lesssim 1$. Hence, by Theorem \ref{GSresult}, we can extend our solution $(f,E,B)$ to a larger time interval $[0,T+\epsilon]$.

\pagebreak

\end{document}